\documentclass[a4paper, 10pt,DIV=10,headinclude=false, footinclude=false]{scrartcl}

\usepackage{amsthm, amsmath, amssymb, amsfonts}
\usepackage[utf8]{inputenc}
\usepackage[english]{babel}
\usepackage{url}
\usepackage{mathtools}

\usepackage{tikz}
\usepackage{tikz-3dplot}
\tdplotsetmaincoords{70}{110}
\usepackage{pgfplots}
\usepackage{pgfplotstable}
\usepackage{booktabs}

\usepackage{rotating}

\numberwithin{figure}{section}
\numberwithin{table}{section}
\numberwithin{equation}{section}

\newenvironment{abstr}[1]{ \vspace{.05in}\footnotesize
       \parindent .2in
         {\upshape\bfseries #1. }\ignorespaces}{\par\vspace{.1in}}
\newenvironment{Abstract}{\begin{abstr}{Abstract}}{\end{abstr}}
\newenvironment{keywords}{\begin{abstr}{Key words}}{\end{abstr}}
\newenvironment{AMS}{\begin{abstr}{AMS subject
      classifications}}{\end{abstr}}

\newtheorem{theorem}{Theorem}[section]
\newtheorem{lemma}[theorem]{Lemma}

\theoremstyle{definition}

\DeclareMathOperator{\diag}{diag}
\DeclareMathOperator{\diam}{diam}
\DeclareMathOperator{\Div}{div}
\DeclareMathOperator{\curl}{curl}

\DeclareMathOperator{\eff}{eff}
\DeclareMathOperator{\Id}{Id}

\DeclareMathOperator{\HMM}{HMM}
\DeclareMathOperator{\imp}{imp}
\renewcommand{\Re}{\operatorname{Re}}
\renewcommand{\Im}{\operatorname{Im}}
\newcommand{\norm}[2]{\left\| #2 \right\|_{#1}}
\DeclareMathOperator*{\wto}{\rightharpoonup}
\newcommand{\gz}{\mathbb{Z}}       
\newcommand{\rz}{\mathbb{R}}       
\newcommand{\cz}{\mathbb{C}}       
\newcommand\Ve{\mathbf{e}}

\newcommand\CT{\mathcal{T}}

\usepackage{changes}
\usepackage[textsize=footnotesize]{todonotes}

\usepackage[colorlinks=true,linkcolor=black,citecolor=black]{hyperref}
\usepackage[labelformat=simple]{subcaption}
\usepackage[leftcaption]{sidecap} 
\usepackage{empheq}
\usepackage{csquotes}

\usepackage{dsfont}
\newcommand{\indicator}{\mathds{1}}
\let\olddelta=\delta
\renewcommand{\delta}{\eta}
\newcommand{\e}{\operatorname{e}}
\newcommand{\R}{\mathbb{R}}
\newcommand{\C}{\mathbb{C}}
\newcommand{\Q}{Q}
\usepackage{enumitem}

\newcommand{\Eeff}{\hat{E}}
\newcommand{\Heff}{\hat{H}}
\renewcommand{\d}{{\, \operatorname{d}}}
\newcommand{\weakto}{\rightharpoonup}
\newcommand{\eps}{\varepsilon}

\usepackage{etoolbox}                  
\DeclarePairedDelimiterX\set[1]\lbrace\rbrace{\def\given{\colon}#1}                            

\DeclarePairedDelimiterX\abs[1]\lvert\rvert{
\ifblank{#1}{\:\cdot\:}{#1}}

\DeclarePairedDelimiterX\normm[1]\lVert\rVert{
\ifblank{#1}{\:\cdot\:}{#1}}

\allowdisplaybreaks[4]

\begin{document}

\title{Mathematical analysis of transmission properties of
  electromagnetic meta-materials%
  \thanks{This work was supported by the Deutsche
    Forschungsgemeinschaft (DFG) in the project ``Wellenausbreitung in
    periodischen Strukturen und Mechanismen negativer Brechung''
    (grant OH 98/6-1 and SCHW 639/6-1).}  } \author{M.
  Ohlberger\footnotemark[2] \and B. Schweizer\footnotemark[3] \and
  M. Urban\footnotemark[3] \and B. Verf\"urth\footnotemark[2]}
\date{September 21, 2018}
\maketitle

\renewcommand{\thefootnote}{\fnsymbol{footnote}}
\footnotetext[2]{Angewandte Mathematik: Institut f\"ur Analysis und
  Numerik, Westf\"alische Wilhelms-Uni\-ver\-si\-t\"at M\"unster,
  Einsteinstr. 62, DE-48149 M\"unster} \footnotetext[3]{Fakult\"at
  f\"ur Mathematik, TU Dortmund, Vogelspothsweg 87, DE-44227 Dortmund}
\renewcommand{\thefootnote}{\arabic{footnote}}

\begin{Abstract}
  We study time-harmonic Maxwell's equations in meta-materials that
  use either perfect conductors or high-contrast materials.  Based on
  known effective equations for perfectly conducting inclusions, we
  calculate the transmission and reflection coefficients for four
  different geometries.  For high-contrast materials and essentially
  two-dimensional geometries, we analyze parallel electric and
  parallel magnetic fields and discuss their potential to exhibit
  transmission through a sample of meta-material. For a numerical
  study, one often needs a method that is adapted to heterogeneous
  media; we consider here a Heterogeneous Multiscale Method for high
  contrast materials.  The qualitative transmission properties, as
  predicted by the analysis, are confirmed with numerical
  experiments. The numerical results also underline the applicability
  of the multiscale method.
\end{Abstract}



\section{Introduction}
\label{sec:introduction}

\textbf{Motivation.}  We study the transmission and reflection
properties of meta-materials, i.e., of periodic microstructures of a
composite material with two components.  The interest in
meta-materials has immensely grown in the last years as they exhibit
astonishing properties such as band gaps or negative refraction; see
\cite{EP04negphC, PE03lefthanded, CJJP02negrefraction}.  The
propagation of electromagnetic waves in such materials is modelled by
time-harmonic Maxwell's equations for the electric field $E$ and the
magnetic field $H$:
\begin{subequations}
  \begin{empheq}[left=\empheqlbrace]{alignat=2}
    \curl E &= \phantom{-} i \omega \mu_0\mu H\,, \\
    \curl H &= - i \omega \varepsilon_0\varepsilon E\,.
  \end{empheq}
\end{subequations}
We use the standard formulation with $\mu_0, \eps_0 > 0$ the
permeability and permittivity of vacuum, $\mu$ and $\eps$ the
corresponding relative parameters, and $\omega>0$ the imposed frequency.
While most materials are non-magnetic, i.e., $\mu=1$, the electric
permittivity $\varepsilon$ covers a wide range.  In this paper, we
study meta-materials consisting of air (i.e., $\varepsilon=1$) and a
(metal) microstructure $\Sigma_\eta$.  The microstructure is
assumed to be an $\eta$-periodic repetition of scaled copies of some
geometry $\Sigma$. In the present study, we investigate in detail four
different geometries: $\Sigma$ can be a metal cylinder (in two
rotations), a metal plate, or the complement of an air cylinder;
see Fig. \ref{fig:Two-examples-of-perfect-conductors-for-analysis}
and \eqref {eq:Analysis-Definition-of-metal-cylinder}--\eqref
{eq:Analysis-Definition-of-air-cylinder} for a detailed definition.
For the electric permittivity in the microstructure $\Sigma_\eta$, we
consider two different cases: perfect conductors that are formally
obtained by setting $\varepsilon=\infty$, and high-contrast materials
with $\varepsilon=\varepsilon_1\eta^{-2}$, where $\eps_1\in \C$ is
some complex number with $\Im(\eps_1) > 0$.  In both cases, our study
is based on the effective equations for the electric and magnetic
field in the limit $\eta \to 0$.

The numerical simulation of electromagnetic wave propagation in such
meta-materials is very challenging because of the rapid variations in
the electric permittivity.  Standard methods require the resolution of
the $\eta$-scale, which often becomes infeasible even with today's
computational resources.  Instead, we resort to
homogenization and multiscale methods to extract macroscopic features
and the behaviour of the solution.  The effective equations obtained
by homogenization can serve as a good motivation and
starting point in this process.

\smallskip \textbf{Literature.}  Effective equations for Maxwell's
equations in meta-materials are obtained in several different settings
with various backgrounds in mind: Dielectric bulk inclusions with
high-contrast media \cite{BBF09hommaxwell, BBF15hommaxwell,
  CC15hommaxwell} can explain the effect of artificial magnetism and
lead to unusual effective permeabilities $\mu$, while long wires
\cite{BB12homwire} lead to unusual effective permittivities.  A
combination of both structures is used to obtain a negative-index
meta-material in \cite{LS15negindex}.  Topological changes in the
material in the limit $\eta\to 0$, such as found in split rings
\cite{BS10splitring}, also incite unusual effective behaviour.
Perfect conductors were recently studied as well: split rings in
\cite{LS16pecrings} and different geometries in \cite{SU17hommaxwell}.
Finally, we briefly mention that the Helmholtz equation---as the
two-dimensional reduction of Maxwell's equations---is often studied as
the first example for unusual effective properties: high-contrast
inclusions in \cite{BF04homhelmholtz} or high-contrast layer materials
in \cite{BS13plasmonwaves}, just to name a few.  An overview on this
vast topic is provided in \cite{Schw17metamaterial}.

Concerning the numerical treatment, we focus on the Heterogeneous
Multiscale Method (HMM) \cite{EE03hmm, EE05hmm}.  For the HMM, first
analytical results concerning the approximation properties for
elliptic problems have been derived in \cite{Abd05hmmanalysis,
  EMZ05hmmanalysis, Ohl05HMM} and then extended to other problems,
such as time-harmonic Maxwell's equations \cite{HOV15maxwellHMM} and
the Helmholtz equation and Maxwell's equations with high-contrast
\cite{OV16hmmhelmholtz, Ver17hmmmaxwell}.  Another related work is the
multiscale asymptotic expansion for Maxwell's equations
\cite{CZAL10multiscalemaxwell}.
For further recent contributions to HMM approximations for
Maxwell's equations we refer to \cite{CFS17hmmmaxwell,HS16hmmmaxwelltime}.
Sparse tensor product finite elements for multiscale Maxwell-type equations
are analyzed in \cite{CH17femmaxwellmultiscale} and an adaptive generalized
multiscale finite element method is studied in \cite{CL19gmsfem}.

\smallskip \textbf{Main results.} We perform an analytical and a
numerical study of transmission properties of meta-materials that
contain either perfect conductors or high-contrast materials.  The
main results are the following:

1.) Using the effective equations of \cite{SU17hommaxwell}, we
calculate the reflection and transmission coefficients for four
microscopic geometries $\Sigma$.  Few geometrical parameters are
sufficient to fully describe the effective coefficients.  We show that
only certain polarizations can lead to transmission.

2.) For the two geometries that are invariant in the $\e_3$-direction,
we study the limit behaviour of the electromagnetic fields for
high-contrast media. When the electric field is parallel to $\e_3$,
all fields vanish in the limit. Instead, when the magnetic field is
parallel to $\e_3$, transmission cannot be excluded due to resonances.

3.) Extensive numerical experiments for high-contrast media confirm
the analytical results.  The numerical experiments underline the
applicability of the Heterogeneous Multiscale Method to these
challenging setting.

Some further remarks on 2.) are in order. The results are related to
homogenization results of \cite{BBF15hommaxwell, CC15hommaxwell}, but
we study more general geometries, since the highly conducting material
can be connected. Furthermore, the results are related to
\cite{BF04homhelmholtz, BS13plasmonwaves}, where connected structures
are investigated, but in a two-dimensional formulation. We treat here
properties of the three-dimensional solutions.  We emphasize that the
transmission properties of a high-contrast medium cannot be captured
in the framework of perfect conductors, since the latter excludes
resonances on the scale of the periodicity (except if three different
length-scales are considered as in \cite {LS16pecrings}).

\smallskip \textbf{Organization of the paper.}  The paper is organized
as follows: In Section \ref{sec:problem}, we detail the underlying
problem formulations and revisit existing effective equations.  In
Section \ref{sec:analysis}, we compute the transmission coefficients
for perfect conductors and derive effective equations for
high-contrast media.  In Section \ref{sec:numerics}, we briefly
introduce the Heterogeneous Multiscale Method. Finally, in Section
\ref{subsec:experiment} we present several numerical experiments
concerning the transmission properties of our geometries for
high-contrast materials.

\section{Problem formulation and effective equations}
\label{sec:problem}

This section contains the precise formulation of the problem,
including the description of the four microscopic geometries.  We
summarize the relevant known homogenization results and apply them to
the cases of interest.

\subsection{Geometry and material parameters}
\label{sec:geometry}
We study the time-harmonic Maxwell equations with linear material
laws. The geometry is periodic with period $\eta>0$; solutions depend
on this parameter and are therefore indexed with $\eta$. On a domain
$G \subset \R^3$, the problem is to find $E^\eta, H^\eta :G \to \C^3$,
such that
\begin{subequations}\label{eq:time-harmonic-Maxwell-eq}
  \begin{empheq}[left=\empheqlbrace]{alignat=2}
    \curl E^\eta& = \phantom{-} i \omega \mu_0 H^\eta \,,
    \label{eq:time-harmonic-Maxwell-eq1}\\
    \curl H^\eta&= - i \omega \varepsilon_0\varepsilon_\eta E^\eta \,,
    \label{eq:time-harmonic-Maxwell-eq2}
  \end{empheq}
\end{subequations}
subject to appropriate boundary conditions.  In the following, we will
give details on the geometry $G$ and on the choice of the material
parameter $\varepsilon_\eta$, the relative permittivity.  Note that
the system allows to eliminate one unknown. Indeed, if we insert
$H^\eta$ from \eqref{eq:time-harmonic-Maxwell-eq1} into
\eqref{eq:time-harmonic-Maxwell-eq2}, we obtain
\begin{equation}
  \curl\curl E^\eta = \omega^2 \mu_0\varepsilon_0\varepsilon_\eta E^\eta\,.
  \label{eq:time-harmonic-Maxwell-eq-E}
\end{equation}
Alternatively, substituting $E^\eta$ from \eqref
{eq:time-harmonic-Maxwell-eq2} into \eqref
{eq:time-harmonic-Maxwell-eq1}, we obtain
\begin{equation}
  \curl \varepsilon_\eta^{-1}\curl H^\eta = \omega^2 \mu_0\varepsilon_0
  H^\eta\,.\label{eq:time-harmonic-Maxwell-eq-H}
\end{equation}

\begin{figure}
  \centering
  \includegraphics[width=.7\linewidth]{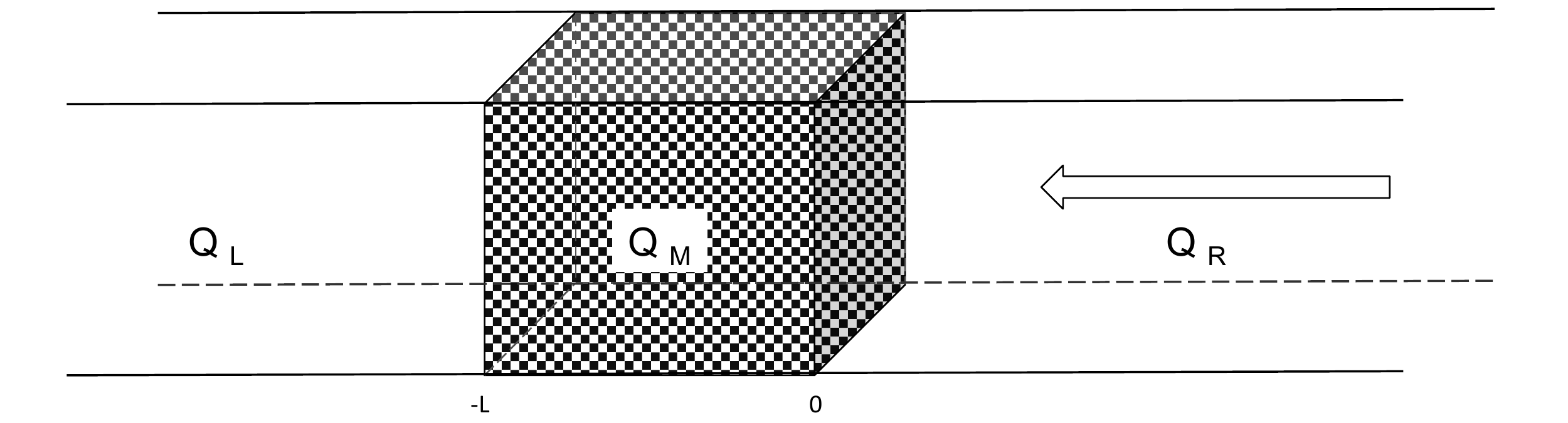}
  \caption{Waveguide domain $G$ with periodic scatterer $\Sigma_{\delta}$
  contained in the middle part $\Q_M$ and incident wave from the right.
  \label{fig:waveguide}}
\end{figure}

\textbf{Geometry.}  As sketched in Fig. \ref{fig:waveguide}, with positive numbers $\ell_2, \ell_3 > 0$, the
unbounded macroscopic domain is the waveguide domain
\begin{equation}\label{eq:Analysis-Definition-of-G}
  G \coloneqq \set[\big]{x =(x_1, x_2, x_3) \in \R^3 \given x_2 \in (-
    \ell_2, \ell_2) \text{ and } x_3 \in (- \ell_3, \ell_3)}\,.
\end{equation}
With another positive number $L > 0$, the domain is divided into three
parts (left, middle, right) as
\begin{equation*}
  \Q_L \coloneqq \set[\big]{x \in G \given x_1 \leq -L}\, ,\ \Q_M
  \coloneqq \set[\big]{x \in G \given x_1 \in (-L, 0)}\, , \  \text{
    and } \ \Q_R \coloneqq \set{x \in G \given x_1 \geq 0}\, .
\end{equation*}
The scatterer $\Sigma_{\delta}$ is contained in the middle part
$\Q_M$.  For the boundary conditions, we consider an incident wave
from the right that travels along the $x_1$-axis to the left.  We
restrict ourselves here to normal incidence.  For the analysis, we
impose periodic boundary conditions on the lateral boundaries of the
domain $G$.  For the numerics, we will modify the boundary conditions
slightly: we truncate $G$ in $x_1$-direction (to obtain a bounded
domain) and consider impedance boundary conditions (with the incident
wave as data) on the whole boundary of $G$.

\begin{figure}
  \centering
\begin{subfigure}[t]{0.3\textwidth}
  \begin{tikzpicture}[scale=2.5]
  \coordinate (A) at (-0.5, -0.5, 0.5);
  \coordinate (B) at (0.5,-0.5,0.5);
  \coordinate (C) at (0.5,0.5,0.5);
  \coordinate (D) at (-0.5, 0.5, 0.5);

  \coordinate (E) at (-0.5, 0.5, -0.5);
  \coordinate (F) at (0.5, 0.5, -0.5);
  \coordinate (G) at (0.5, -0.5, -0.5);
  \coordinate (H) at (-0.5, -0.5, -0.5);

  \fill[gray!20!white, opacity=.5] (A) -- (B) -- (C) -- (D) -- cycle;
  \fill[gray!20!white, opacity=.5] (E) -- (F) -- (G) --(H) -- cycle;
  \fill[gray!20!white, opacity=.5] (D) -- (E) -- (H) -- (A)-- cycle;
  \fill[gray!20!white, opacity=.5] (B) -- (C) -- (F) -- (G) -- cycle;

  \fill[gray!75!white] (-.25, -.5) -- (-.25,.5)
                                arc(180:0:0.25cm and 0.125cm) -- (0.25, -0.5)
                                (-0.25, -.5) arc(180:360:0.25cm and 0.125cm);
  \filldraw[gray!75!white] (-.25,-.5) -- (0.25,-.5);

  \draw[] (A) -- (B) -- (C) -- (D) --cycle (E) -- (D) (F) -- (C) (G) -- (B);
  \draw[] (E) -- (F) -- (G) ;
  \draw[densely dashed] (E) -- (H) (H) -- (G) (H) -- (A);

  \draw[] (-.25,0.5) arc (180:-180:0.25cm and 0.125cm);
  \draw[dashed] (-.25,-.5) arc (180:360:0.25cm and 0.125cm);
  \draw[dashed] (-.25,-.5) arc (180:0:0.25cm and 0.125cm);
  \draw[dashed] (-0.25, -.5) -- (-0.25, .5);
  \draw[dashed] (.25,-.5) -- (.25,.5);

  \draw[->] ( -.8, -.7, .5)--( -.4, -.7, .5 );
  \draw[->] (-.8, -.7, .5) -- (-.8, -.3, .5);
  \draw[->] (-.8, -.7, .5) -- (-.8, -.7, .2);

  \node[] at (0,0,0) {$\Sigma_1$};
\node[] at (-.4, -.8, .5){$\e_1$};
  \node[] at (-.68, -.3, .5){$\e_3$};
  \node[] at (-.8, -.8, .-.1){$\e_2$};

\end{tikzpicture}
  \caption{ }
  \label{fig:the metal cylinder}
\end{subfigure}
\hfill
\begin{subfigure}[t]{0.3\textwidth}
  \begin{tikzpicture}[scale=2.5]
  \coordinate (A) at (-0.5, -0.5, 0.5);
  \coordinate (B) at (0.5,-0.5,0.5);
  \coordinate (C) at (0.5,0.5,0.5);
  \coordinate (D) at (-0.5, 0.5, 0.5);

  \coordinate (E) at (-0.5, 0.5, -0.5);
  \coordinate (F) at (0.5, 0.5, -0.5);
  \coordinate (G) at (0.5, -0.5, -0.5);
  \coordinate (H) at (-0.5, -0.5, -0.5);

  \coordinate (M1) at (-0.5, -0.5, 0.15);
  \coordinate (M2) at (-0.5, -0.5, -0.15);
  \coordinate (M3) at (-0.5, 0.5, -0.15);
  \coordinate (M4) at (-0.5, 0.5, 0.15);
  \coordinate (M5) at (0.5, 0.5, 0.15);
  \coordinate (M6) at (0.5,0.5,-0.15);
  \coordinate (M7) at (0.5, -0.5, -0.15);
  \coordinate (M8) at (0.5, -0.5, 0.15);

  \fill[gray!20!white, opacity=.5] (A) -- (B) -- (C) -- (D) -- cycle;
  \fill[gray!20!white, opacity=.5] (E) -- (F) -- (G) --(H) -- cycle;
  \fill[gray!20!white, opacity=.5] (D) -- (E) -- (H) -- (A)-- cycle;
  \fill[gray!20!white, opacity=.5] (B) -- (C) -- (F) -- (G) -- cycle;

  \fill[gray!75!white, opacity=.9] (M4) -- (M5) -- (M6) -- (M3) -- cycle;
  \fill[gray!75!white, opacity=.9] (M1) -- (M4) -- (M5) -- (M8) -- cycle;
  \fill[gray!75!white, opacity=.9] (M2) -- (M3) -- (M6) -- (M7) -- cycle;
  \fill[gray!75!white, opacity=.9] (M1) -- (M2) -- (M7) -- (M8) -- cycle;

  \draw[] (A) -- (B) -- (C) -- (D) --cycle (E) -- (D) (F) -- (C) (G) -- (B);
  \draw[] (E) -- (F) -- (G) ;
  \draw[densely dashed] (E) -- (H) (H) -- (G) (H) -- (A);

  \draw[dashed] (M1) -- (M4) (M2) -- (M3) ;
  \draw[] (M4) -- (M5) -- (M6) -- (M3) -- cycle;
  \draw[dashed] (M5) -- (M8) (M7) -- (M6);
  \draw[dashed] (M1) -- (M8)  (M7) -- (M2);

  \draw[->] ( -.8, -.7, .5)--( -.4, -.7, .5 );
  \draw[->] (-.8, -.7, .5) -- (-.8, -.3, .5);
  \draw[->] (-.8, -.7, .5) -- (-.8, -.7, .2);

  \node[] at (0, 0, 0) {$\Sigma_3$};
  \node[] at (-.4, -.8, .5){$\e_1$};
  \node[] at (-.68, -.3, .5){$\e_3$};
  \node[] at (-.8, -.8, .-.1){$\e_2$};
\end{tikzpicture}
  \caption{ }
  \label{fig:the metal plate}
\end{subfigure}
\hfill\begin{subfigure}[t]{0.3\textwidth}
    \begin{tikzpicture}[scale=2.5]
  \coordinate (A) at (-0.5, -0.5, 0.5);
  \coordinate (B) at (0.5,-0.5,0.5);
  \coordinate (C) at (0.5,0.5,0.5);
  \coordinate (D) at (-0.5, 0.5, 0.5);

  \coordinate (E) at (-0.5, 0.5, -0.5);
  \coordinate (F) at (0.5, 0.5, -0.5);
  \coordinate (G) at (0.5, -0.5, -0.5);
  \coordinate (H) at (-0.5, -0.5, -0.5);

  \fill[gray!75!white, opacity=.9] (A) -- (B) -- (C) -- (D) -- cycle;
  \fill[gray!75!white, opacity=.9] (E) -- (F) -- (G) --(H) -- cycle;
  \fill[gray!75!white, opacity=.9] (D) -- (E) -- (H) -- (A)-- cycle;
  \fill[gray!75!white, opacity=.9] (B) -- (C) -- (F) -- (G) -- cycle;

  \fill[gray!20!white] (-0.5, 0.25, 0) -- (.5,.25,0)
                                arc(90:-90:0.125cm and 0.25cm) --
                                (0.5, -.25,0) --
                                (-0.5, -0.25,0) arc(-90:-270:0.125cm and 0.25cm);

  \draw[] (A) -- (B) -- (C) -- (D) --cycle (E) -- (D) (F) -- (C) (G) -- (B);
  \draw[] (E) -- (F) -- (G) ;
  \draw[densely dashed] (E) -- (H) (H) -- (G) (H) -- (A);

  \draw[] (0.375, 0,0) arc (-180:180:0.125cm and 0.25cm);
  \draw[dashed] (-.625,0, 0) arc (180:360:0.125cm and 0.25cm);
  \draw[dashed] (-.625,0,0) arc (180:0:0.125cm and 0.25cm);
  \draw[dashed] (-0.5, 0.25,0) -- (0.5, 0.25, 0);
  \draw[dashed] (-0.5,-.25,0) -- (.5,-.25,0);

  \draw[->] ( -.8, -.7, .5)--( -.4, -.7, .5 );
  \draw[->] (-.8, -.7, .5) -- (-.8, -.3, .5);
  \draw[->] (-.8, -.7, .5) -- (-.8, -.7, .2);

  \node[] at (0, -.4, .4) {$\Sigma_4$};
  \node[] at (-.4, -.8, .5){$\e_1$};
  \node[] at (-.68, -.3, .5){$\e_3$};
  \node[] at (-.8, -.8, .-.1){$\e_2$};
\end{tikzpicture}
  \caption{ }
  \label{fig:the air cylinder}
\end{subfigure}
\caption{The cube shows the periodicity cell $Y$. The microstructures
  $\Sigma_1$, $\Sigma_3$, and $\Sigma_4$ are shown in dark grey.  (a)
  The metal cylinder $\Sigma_1$. (b) The metal plate $\Sigma_3$. (c)
  The metal part $\Sigma_4$ is the complement of a cylinder.
\label{fig:Two-examples-of-perfect-conductors-for-analysis}}
\end{figure}
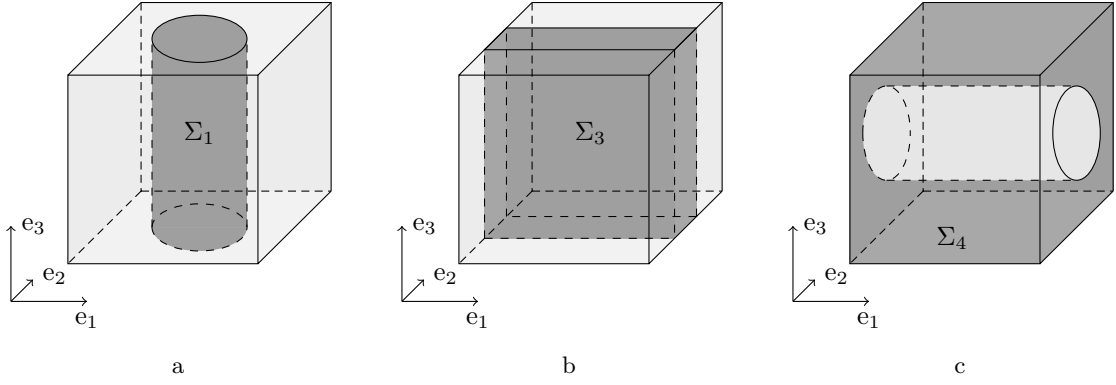

The scatterer $\Sigma_\eta$ is given as an $\eta$-periodic structure.
We use the periodicity cell $Y \coloneqq [-\frac 12 ,\frac 12]^3$ and
introduce the set $I_\eta$ of all vectors such that a scaled and shifted
copy of $Y$ is contained in $Q_M$, $I_\eta \coloneqq \{j\in
\gz^3|\eta(j+Y)\subset Q_M\}$.  A set $\Sigma\subset Y$ specifies the
meta-material, which is defined as
\begin{equation}
  \Sigma_\eta \coloneqq \bigcup_{j\in
    I_\eta}\ \eta\, (j+\Sigma)\,.\label{eq:Sigma-eta}
\end{equation}

For the microscopic structure $\Sigma$ we consider the following four
examples.  The metal cylinder (see Fig.~\ref{fig:the metal
  cylinder}) is defined for $r \in (0, 1/2)$ as
\begin{equation}\label{eq:Analysis-Definition-of-metal-cylinder}
  \Sigma_{1} \coloneqq \set[\big]{y = (y_1, y_2, y_3) \in Y \given y_1^2 + y_2^2 < r^2}\, .
\end{equation}
The set $\Sigma_2$ is obtained by a rotation which aligns the cylinder
with the $\e_1$-axis,
\begin{equation}\label{eq:Analysis-Definition-of-metal-cylinder-rotated}
  \Sigma_2 \coloneqq \set[\big]{y = (y_1, y_2, y_3) \in Y \given
    y_2^2+y_3^2 < r^2}\, .
\end{equation}
To define the metal plate (see Fig.~\ref{fig:the metal plate}), we
fix $r \in (0, 1/2)$ and set
\begin{equation}\label{eq:Analysis-Definition-of-rotated-metal-plate}
  \Sigma_3 \coloneqq \set[\big]{y = (y_1, y_2, y_3)
    \in Y \given y_2 \in (- r, r)}\, .
\end{equation}
The fourth geometry is obtained by removing an \enquote{air cylinder} from
the unit cube (see Fig.~\ref{fig:the air cylinder}); for
$r \in (0, 1/2)$ we set
\begin{equation}\label{eq:Analysis-Definition-of-air-cylinder}
  \Sigma_4 \coloneqq Y \setminus  \set[\big]{y = (y_1, y_2, y_2) \in Y
    \given y_2^2 + y_3^2 < r^2}\,.
\end{equation}

\textbf{Material parameters.}  We recall that all materials are
non-magnetic, the relative magnetic permeability is $\mu \equiv
1$. Outside the central region, there is no scatterer; we hence set
$\varepsilon_\eta=1$ in $Q_L$ and $Q_R$.  The middle part $Q_M$
contains $\Sigma_\eta$. We set $\varepsilon_\eta = 1$ in $Q_M\setminus
\Sigma_\eta$.  It remains to specify the electric permittivity
$\varepsilon_\eta$ in $\Sigma_\eta$. We consider two different
settings.

\smallskip (PC) In the case of \emph{perfect conductors}, we set,
loosely speaking, $\varepsilon_\eta = +\infty$ in $\Sigma_\eta$.  More
precisely, we require that $E^\eta$ and $H^\eta$ satisfy
\eqref{eq:time-harmonic-Maxwell-eq} in $G\setminus
\overline{\Sigma}_\eta$ and $E^\eta=H^\eta=0$ in
$\Sigma_\eta$. Boundary conditions are induced on $\partial
\Sigma_\eta$: The magnetic field $H^\eta$ has a vanishing normal
component and the electric field $E^\eta$ has vanishing tangential
components on $\partial \Sigma_\eta$.

\smallskip (HC) In the case of \emph{high-contrast media}, we define
the permittivity as
\begin{equation}\label{eq:Analysis-high-contrast-permittivity}
  \varepsilon_\eta(x) \coloneqq
  \begin{dcases}
    \frac{\varepsilon_1}{\eta^2}& \text{ if } x \in \Sigma_\eta\, ,\\
    1 & \text{ if } x \in G \setminus \Sigma_\eta\, ,
  \end{dcases}
\end{equation}
where $\varepsilon_1 \in \mathbb{C}$ with $\Re(\varepsilon_1) >0$,
$\Im(\varepsilon_1)>0$.  Physically speaking, this means that the
scatterer $Q_M$ consists of periodically disposed metal inclusions
$\Sigma_\delta$ embedded in vacuum.  The scaling with $\eta^2$ means
that the optical thickness of the inclusions remains constant; see
\cite{BBF15hommaxwell}.

\smallskip In both settings and throughout this paper, we consider
sequences of solutions $(E^\eta, H^\eta)_{\eta}$ to
\eqref{eq:time-harmonic-Maxwell-eq} which are bounded in $L^2(G;
\mathbb{C}^3)$,
\begin{equation}
  \label{eq:Analysis-high-contrast-boundedness-assumption-on-fields}
  \sup_{\eta > 0} \int_G \abs{E^\eta}^2 + \abs{H^\eta}^2 < \infty\, .
\end{equation}

Let us remark that the specific geometry of the microstructures $\Sigma_1,
\Sigma_2$, and $\Sigma_4$
is not important; the cylinders could as well be cuboids.

\subsection{Effective equations}
\label{sec:homogenization}
Homogenization theory allows to consider the limit $\eta\to 0$. One
identifies limiting fields $\hat{E}$ and $\hat{H}$ (the latter does
not coincide with the weak limit of $H^\eta$) and limiting equations
for these fields.  Using the tool of two-scale convergence, such
results have been obtained for perfect conductors as well as for
high-contrast materials.  We briefly summarize the main findings here;
analysis and numerics below are built upon these results.

\smallskip \textbf{Perfect conductors (PC).}  The homogenization analysis for
this case has been performed in \cite{SU17hommaxwell}.  Since the
parameters of vacuum are used outside the scatterer, the original
Maxwell equations describe the limiting fields in $Q_L$ and $Q_R$.  In
the meta-material $Q_M$, however, different equations hold. There
holds $E^\eta\wto\Eeff$ and $H^\eta \wto \hat{\mu}\Heff$ in $L^2(G)$
and the fields $\Eeff$ and $\Heff$ solve
\begin{subequations}\label{eq:Analysis-the-general-effective-equations}
  \begin{empheq}[left=\empheqlbrace]{alignat=2}
    \curl \Eeff & = \phantom{-} i \omega \mu_0 \hat{\mu} \Heff &&
    \text{ in } G\,
    , \label{eq:Analysis-the-general-effective-equations-1}\\
    \curl \Heff & = - i \omega \varepsilon_0 \hat{\varepsilon} \Eeff
    && \text{ in } G \setminus Q_M \,
    , \label{eq:Analysis-the-general-effective-equations-2}\\
    (\curl \Heff)_k & = -i \omega \varepsilon_0 (\hat{\varepsilon }
    \Eeff)_k && \text{ in } G\,  , \text{ for every } k \in
    \mathcal{N}_{\Sigma}\,
    , \label{eq:Analysis-the-general-effective-equations-3}\\
    \Eeff_k & = 0 && \text{ in } Q_M \, , \text{ for every } k \in
    \mathcal{L}_{\Sigma}\,
    , \label{eq:Analysis-the-general-effective-equations-4}\\
    \Heff_k & = 0 && \text{ in } Q_M \, , \text{ for every } k \in
    \mathcal{N}_{Y \setminus \overline{\Sigma}}\, .
    \label{eq:Analysis-the-general-effective-equations-5}
  \end{empheq}
\end{subequations}
The effective coefficients $\hat{\mu}$ and $\hat{\varepsilon}$ are
determined by cell-problems.  For the cell-problems, details on the
index sets, and the derivation of system
\eqref{eq:Analysis-the-general-effective-equations}, we refer to
\cite{SU17hommaxwell}.  The index sets $\mathcal{N}_{\Sigma}$,
$\mathcal{L}_{\Sigma}$, and $\mathcal{N}_{Y \setminus
  \overline{\Sigma}}$ are subsets of $\set{1,2,3}$ and can be
determined easily from topological properties of $\Sigma$. Loosely
speaking: An index $k$ is in the set $\mathcal{L}_{\Sigma}$, if there
is a curve (loop) that runs in $\Sigma$ and connects opposite faces of
$Y$ in direction $\e_k$. An index $k$ is in $\mathcal{N}_{\Sigma}$, if
there is no loop of that kind.  We collect the index sets
$\mathcal{N}_{\Sigma}$, $\mathcal{L}_{\Sigma}$, and $\mathcal{N}_{Y
  \setminus \overline{\Sigma}}$ for the geometries $\Sigma_1$ to
$\Sigma_4$ in Table \ref{tab:analysissummary}.

\begin{table}
  \caption{Index sets $\mathcal{N}_{\Sigma}$, $\mathcal{L}_{\Sigma}$,
    and $\mathcal{N}_{Y \setminus \overline{\Sigma}}$ for
    microstructures $\Sigma_1$ to $\Sigma_4$ of
    \eqref{eq:Analysis-Definition-of-metal-cylinder}--\eqref
    {eq:Analysis-Definition-of-air-cylinder}.}
  \label{tab:analysissummary}
\centering
\begin{tabular}{@{}ccccc@{}}
\toprule
geometry&metal cylinder $\Sigma_1$&metal cylinder $\Sigma_2$&metal plate $\Sigma_3$&air cylinder $\Sigma_4$\\
\midrule
$\mathcal{N}_{\Sigma}$& $\{1, 2\}$ & $\{2,3\}$ & $\{2\}$ & $\emptyset$ \\
\midrule
$\mathcal{L}_{\Sigma}$& $\{3\}$ & $\{1\}$ & $\{1,3\}$ & $\{1,2,3\}$ \\
\midrule
$\mathcal{N}_{Y \setminus \overline{\Sigma}}$& $\emptyset$ & $\emptyset$ & $\{2\}$ & $\{2,3\}$ \\
\bottomrule
\end{tabular}
\end{table}

We will specify equations \eqref
{eq:Analysis-the-general-effective-equations-3}--\eqref
{eq:Analysis-the-general-effective-equations-5} for the four chosen
geometries in Section \ref
{sec:Calculation-of-transmission-and-reflection-coefficient}.  With
the effective equations for the perfect conductors at hand, one can
ask for the transmission and reflection coefficients of the
meta-material.  This is the goal of our analysis in Section
\ref{sec:Calculation-of-transmission-and-reflection-coefficient}.

\smallskip \textbf{high-contrast media (HC).} Homogenization results
for high-contrast media are essentially restricted to the case of
non-connected metal parts, i.e., to geometries that are obtained by
$\Sigma$ which is compactly embedded in $Y$ (it does not touch the
boundary of the cube); see, e.g., \cite {BBF09hommaxwell,
  BS10splitring, BBF15hommaxwell}. The few exceptions are mentioned
below.

For such geometries, the limit equations have again the form of
Maxwell's equations,
\begin{subequations}\label{eq:Analysis-high-contrast-effective-equations}
  \begin{empheq}[left=\empheqlbrace]{alignat=2}
    \curl \Eeff & = \phantom{-} i \omega \mu_0 \hat{\mu} \Heff &&
    \text{ in } G\,
    , \label{eq:Analysis-high-contrast-effective-equations-1}\\
    \curl \Heff & = - i \omega \varepsilon_0 \hat{\varepsilon} \Eeff
    && \text{ in } G .
  \end{empheq}
\end{subequations}
In $Q_L\cup Q_R$, the effective fields coincide with the weak limits
of the original fields, and the effective relative coefficients are
unit tensors. In the meta-material $Q_M$, however, the high-contrast
in the definition of the permittivity $\varepsilon_\eta$ in
\eqref{eq:Analysis-high-contrast-permittivity} leads to non-trivial
limit equations. The effective material parameters $\hat{\varepsilon}$
and $\hat{\mu}$ are obtained via cell problems and they can take
values that are not to be expected from the choice of the material
parameters in the $\eta$-problem.

As discussed in Section \ref{sec:geometry}, time-harmonic Maxwell's
equations can equivalently be written as a single second order PDE for
the $H$-field or the $E$-field.  For the $H$-field we obtain
\begin{align}
  \label{eq:Analysis-high-contrast-effective-eq-H}
  \curl \widehat{\varepsilon^{-1}}\curl \Heff=\omega^2\varepsilon_0\mu_0\hat{\mu} \Heff
  \qquad \text{in} \quad G\,.
\end{align}
Again, the effective material parameters $\widehat{\varepsilon^{-1}}$
and $\hat{\mu}$ are defined via solutions of cell problems and we
refer to \cite{CC15hommaxwell, Ver17hmmmaxwell} for details.  We
remark that the equivalence of the two formulations
\eqref{eq:Analysis-high-contrast-effective-equations} and
\eqref{eq:Analysis-high-contrast-effective-eq-H} has been shown in
\cite{Ver17hmmmaxwell}.  In particular, the effective permeability
$\hat{\mu}$ agrees between both formulations and we have the relation
$\widehat{\varepsilon^{-1}}=(\hat{\varepsilon})^{-1}$.

The effective equations
\eqref{eq:Analysis-high-contrast-effective-equations} or
\eqref{eq:Analysis-high-contrast-effective-eq-H} mean that, in the
limit $\eta\to 0$, the meta-material $Q_M$ with high-contrast
permittivity $\varepsilon_\eta$ behaves like a homogeneous material
with permittivity $\hat{\varepsilon}$ and permeability $\hat{\mu}$.
The occurrence of a permeability $\hat{\mu}$ in the effective
equations is striking and this effect is known as artificial
magnetism; see \cite{BF04homhelmholtz}.  Moreover, $\hat{\mu}$ depends
on the frequency $\omega$ and it can have a negative real part for
certain frequencies.  Negative values of the permeability are caused
by (Mie) resonances in the inclusions $\Sigma$ and are studied in
detail in \cite{BBF15hommaxwell, Ver17hmmmaxwell}.

As mentioned, a crucial assumption for the homogenization analysis in
\cite{BBF15hommaxwell, CC15hommaxwell} is that $\Sigma$ is compactly
contained in the unit cube.  For the four geometries
$\Sigma_1$ to $\Sigma_4$, this assumption is clearly not met; we
therefore ask whether certain components of the effective fields
$\hat{E}$ and $\hat{H}$ vanish in this case as in the case of perfect
conductors.  This motivates our analysis in Section \ref
{sec:high-contrast-media} as well as the numerical experiments in
Section \ref{subsec:experiment}.

Regarding known results on non-compactly contained inclusions we
mention the thin wires in \cite {BF-thinwires} and \cite
{LS15negindex}, and the dimensionally reduced analysis of the metal
plates $\Sigma_3$ in \cite {BS13plasmonwaves}.

\section{Analysis of the microscopic geometries $\Sigma_1$ to
  $\Sigma_4$}
\label{sec:analysis}

In Section \ref
{sec:Calculation-of-transmission-and-reflection-coefficient}, we treat
the case of perfect conductors and compute the transmission
coefficients from the effective equations \eqref
{eq:Analysis-the-general-effective-equations}.  In Section \ref
{sec:high-contrast-media}, we treat the case of high-contrast media
and discuss the possibility of nontrivial transmission coefficients.

\subsection{Transmission and reflection coefficients for perfect
  conductors}
\label{sec:Calculation-of-transmission-and-reflection-coefficient}

We compute the transmission and reflection coefficients for four
different geometries: metal cylinders, metal plate, and air cylinder.
We consider the wave guide $G=Q_L\cup \bar Q_M\cup Q_R$ of Section
\ref{sec:geometry} and impose periodic boundary conditions on the
lateral boundary of $G$.  We recall that the four microscopic
structures $\Sigma_1$ to $\Sigma_4$ are defined in
\eqref{eq:Analysis-Definition-of-metal-cylinder}--\eqref
{eq:Analysis-Definition-of-air-cylinder}.  Based on the effective
equations \eqref {eq:Analysis-the-general-effective-equations} for the
perfect conductors, we compute the transmission and reflection
coefficients for these geometries.

\smallskip \textbf{Results for perfect conductors.}  Before we discuss the
examples in detail, we present an overview of the results.  The
propagation of the electromagnetic wave in vacuum is described by the
time-harmonic Maxwell equations
\begin{subequations}\label{eq:Analysis-time-harmonic-Maxwell-eq-in-QR}
  \begin{empheq}[left=\empheqlbrace]{alignat=2}
    \curl \hat{E}& = \phantom{-} i \omega \mu_0 \hat{H}\quad && \text{ in }
    Q_L \cup Q_R\, , \label{eq:Analysis-first-time-harmonic-Maxwell-eq-in-QR} \\
    \curl \hat{H} &= - i \omega \varepsilon_0 \hat{E} && \text{ in }
    Q_L \cup Q_R\, .\label{eq:Analysis-second-time-harmonic-Maxwell-eq-in-QR}
  \end{empheq}
\end{subequations}
For the electromagnetic fields, we use the time-convention $\e^{- i
  \omega t}$.
From~\eqref{eq:Analysis-time-harmonic-Maxwell-eq-in-QR} we deduce that
both fields are divergence-free in $Q_L \cup Q_R$.  We shall assume
that the electric field $\hat{E} \colon G \to \mathbb{C}^3$ in $Q_R$
is the superposition of a normalized incoming wave with normal
incidence and a reflected wave:
\begin{equation}\label{eq:Analysis-general-ansatz-for-electric-field-in-QR}
  \hat{E}(x) \coloneqq \big(\e^{-i k_0 x_1} + R \e^{i k_0 x_1}
  \big) \e_k \, ,
\end{equation}
for $x=(x_1, x_2, x_3) \in Q_R$ and $k \in \set{2,3}$. Here, $R \in \mathbb{C}$
is the reflection coefficient and $k_0 = \omega \sqrt{\varepsilon_0
  \mu_0}$. Note that the electric field $\hat{E}$
in~\eqref{eq:Analysis-general-ansatz-for-electric-field-in-QR} travels
along the $x_1$-axis from right to left.

Due to~\eqref{eq:Analysis-first-time-harmonic-Maxwell-eq-in-QR}, the effective
magnetic field $\hat{H} \colon G \to \mathbb{C}^3$ is given by
\begin{equation}\label{eq:Analysis-general-ansatz-for-magnetic-field-in-QR}
  \hat{H}(x) =  (-1)^{l}\frac{k_0}{\omega \mu_0} \big(\e^{-i k_0 x_1}
  - R \e^{i k_0
    x_1}  \big) \e_l\, ,
\end{equation}
where $l = 2$ if $k = 3$ and $l = 3$ if
$k=2$ and $x \in Q_R$. Equation~\eqref{eq:Analysis-second-time-harmonic-Maxwell-eq-in-QR}
is satisfied in $Q_R$ by our choice of $k_0$.

On the other hand, for the transmitted electromagnetic wave in the
left domain $Q_L$, we make the ansatz
\begin{equation}\label{eq:Analysis-general-ansatz-for-E-and-H-in-QL}
  \hat{E}(x) = T \e^{-i k_0 (x_1+L)}\e_k \quad \text{ and } \quad
  \hat{H}(x) = (-1)^l \frac{k_0}{\omega \mu_0} T \e^{-i k_0 (x_1+L)} \e_l \,,
\end{equation}
where $T\in \C$ is the transmission coefficient. We recall that $L >0$
is the width of the meta-material $Q_M$ and $\{ x_1 = -L\}$ is the
interface between left and middle domain. Since the meta-material in
$Q_M$ can lead to reflections, the transmission coefficient $T \in
\mathbb{C}$ does not necessarily satisfy $\abs{T} = 1$; by energy
conservation there always holds $\abs{T} = 1- \abs{R}$.

Our results are collected in Table \ref {table:res-PC}. The table
lists transmission coefficients for the four geometries in the case
that the incoming magnetic field $H$ is parallel to $\e_3$.

\begin{table}[h]
  \centering
  \begin{tabular}{c  c}
    \toprule
    microstructure $\Sigma$ &
                              transmission coefficient $T$
    \\
    \midrule
    metal cylinder $\Sigma_1$ &
                                $T =4  p_1\sqrt{\alpha\gamma}
                                \Big[(\alpha + \gamma)(1-p_1^2) + 2
                                \sqrt{\alpha \gamma} (1+ p_1^2)\Big]^{-1}$ \\[0.3cm]
    metal cylinder $\Sigma_2$
                            & $T = 4 p_2 \sqrt{\gamma}
                              \Big[(1+ \gamma)(1- p_2^2) + 2 \sqrt{\gamma}(1+p_2^2) \Big]^{-1}$  \\[0.3cm]
    metal plate $\Sigma_3$ &  $ T = 4p_0 \alpha
                             \Big[(1+\alpha^2)(1-p_0^2) + 2 \alpha (1+
                             p_0^2)\Big]^{-1}$ \\[0.3cm]
    air cylinder $\Sigma_4$ & $T =0$ \\
    \bottomrule
  \end{tabular}
  \caption{Overview of the transmission coefficients $T$ when $H$ is parallel
    to $\e_3$. We see, in particular, that $T$ is vanishing for the structure
    $\Sigma_4$, but it is nonzero for the other micro-structures. The constant
    $\gamma \in \mathbb{C}$ depends on
    the microstructure and on solutions to cell problems, and is defined in the
    subsequent sections, $\alpha \coloneqq \abs{ Y \setminus \Sigma}$
    is the volume fraction of air, $L > 0$ is the width of the meta-material $Q_M$. We use
    $k_0 = \omega \sqrt{\varepsilon_0 \mu_0}$ and the numbers $p_0
    \coloneqq \e^{i k_0 L}$, $p_1 \coloneqq p_0 \e^{ i \sqrt{\alpha
        \gamma} L}$, and $p_2 \coloneqq p_0 \e^{ i \sqrt{\gamma}L}$.
    \label{table:res-PC}}
\end{table}

In the remainder of this section we compute the transmission
coefficient $T$ and the reflection coefficient $R$ for the four
microscopic geometries and verify, in particular, the formulas of
Table \ref {table:res-PC}. Moreover, the effective fields in the
meta-material $Q_M$ are determined.

\subsubsection{The metal cylinder $\Sigma_1$}
\label{sec:metal-cylinder}

The metal cylinder $\Sigma_1$ has a high symmetry, which allows to
compute the effective permeability $\hat{\mu}$. To do so, we define
the projection $\pi \colon Y \to \R^2$ onto the first two components,
i.e., $\pi(y_1, y_2, y_3) \coloneqq (y_1, y_2)$.  Moreover, we set
$Y^2 \coloneqq \pi(Y)$ and $\Sigma^2_1 \coloneqq \pi(\Sigma_1)$.

Choose $l \in \set{1,2}$ and denote by $H^l \in L^2(Y; \mathbb{C}^3)$
the distributional periodic solution of
\begin{subequations}\label{eq:Analysis-Cell-problem-for-H-field}
  \begin{empheq}[left=\empheqlbrace]{alignat=2}
    \curl H^l & = 0&& \text{ in } Y \setminus \overline{\Sigma}_1\, ,\\
    \Div H^l & = 0 && \text{ in } Y\, ,\\
    H^l & = 0 && \text{ in } \Sigma_1\, ,
  \end{empheq}
\text{with}
  \begin{equation}
    \oint H^l = \e_l\, .
  \end{equation}
\end{subequations}
The normalization of the last equation is defined in \cite
{SU17hommaxwell}; loosely speaking, the left hand collects values of
line integrals of $H^l$, where the lines are curves in $Y\setminus
\Sigma$ and connect opposite faces of $Y$.  Problem \eqref
{eq:Analysis-Cell-problem-for-H-field} is uniquely solvable by
\cite[Lemma 3.5]{SU17hommaxwell}. Given the field $H^l = (H^l_1,
H^l_2, H^l_3)$ we define the field $h^l \colon Y^2 \to \mathbb{C}^2$
as
\begin{equation}\label{eq:Anaylsis-Definition-of-2-dimensional-H-field}
  h^l(y_1, y_2) \coloneqq \int_0^1 (-H^l_2, H^l_1)(y_1, y_2, y_3) \d y_3\, .
\end{equation}

\begin{lemma}\label{lem:}
  Let $H^l \in L^2(Y; \mathbb{C}^3)$ be the solution of
  \eqref{eq:Analysis-Cell-problem-for-H-field}. Then $h^l \in L^2(Y^2;
  \mathbb{C}^2)$ of \eqref
  {eq:Anaylsis-Definition-of-2-dimensional-H-field} is a
  distributional periodic solution to the two-dimensional problem
  \begin{subequations}
    \begin{empheq}[left=\empheqlbrace]{alignat=2}
      \Div h^l & = 0 && \text{ in } Y^2 \setminus \Sigma_1^2\, ,\\
      \nabla^{\bot} \cdot h^l & = 0 && \text{ in } Y^2\, , \\
      h^l & = 0 && \text{ in } \Sigma_1^2\, .
    \end{empheq}
  \end{subequations}
  Moreover, there exists a potential $\psi \in H_{\sharp}^1(Y^2; \C)$
  such that $h^l = \nabla \psi - \olddelta_{2l}\e_1 +
  \olddelta_{1l}\e_2$.
\end{lemma}

\begin{proof}
  The proof consists of a straightforward calculation.
\end{proof}

The decomposition of $h^l$ allows to determine the effective
permeability $\hat{\mu}$, which, by \cite{SU17hommaxwell}, is given as
\begin{equation}\label{eq:Analysis-effective-permeability}
  \hat{\mu}(x) \coloneqq \mu_{\eff} \indicator_{Q_M}(x) + \Id \indicator_{G \setminus
  \overline{Q_M}}(x)\, ,
\end{equation}
where
\begin{equation}
  \label{eq:Analysis-definition-of-the-effective-permeability}
  (\mu_{\eff})_{kl} \coloneqq \int_Y H^l \cdot \e_k\, .
\end{equation}

\begin{lemma}[Effective permeability for the metal cylinder]
  For the microstructure $\Sigma = \Sigma_1$ the
  permeability $\mu_{\eff}$ is given by
  \begin{equation}\label{eq:Analysis-metal-cylinder-effective-permeability}
    \mu_{\eff} =
     \diag \big( 1, 1, \abs{Y \setminus \Sigma_1} \big) \, .
  \end{equation}
\end{lemma}

\begin{proof}
  To shorten the notation, we write $y^{\prime} \coloneqq (y_1, y_2)
  \in Y^2$. Applying Fubini's theorem and using the decomposition of
  $h^1$, we find that
  \begin{equation*}
    (\mu_{\eff})_{11} = \int_Y H^1 \cdot \e_1 = \int_{Y^2} h^1_2(y^{\prime}) \d y^{\prime}
    = \int_{Y^2} \partial_2 \psi (y^{\prime})  \d y^{\prime}  + \abs{Y^2} = 1\,,
  \end{equation*}
  where, in the last equality, we exploited that $\psi$ is
  $Y^2$-periodic and that $\abs{Y^2} = 1$. A similar computation shows
  that ${(\mu_{\eff})}_{22} = 1$.

  To compute ${(\mu_{\eff})}_{12}$, we note that $h^1_1(y^{\prime})
  = \partial_1 \psi(y^{\prime}) $.  Applying Fubini's theorem, we find
  \begin{equation*}
    {(\mu_{\eff})}_{12} = \int_Y H^1 \cdot \e_2
    = - \int_{Y^2} h^1_1(y^{\prime}) \d y^{\prime}
    = -\int_{Y^2} \partial_1 \psi(y^{\prime}) \d y^{\prime} = 0\,.
  \end{equation*}
  As $h^2_2(y^{\prime}) = \partial_2 \psi(y^{\prime})$, we can proceed
  as before and find ${(\mu_{\eff})}_{21} = 0$.

  One readily checks that $H^3(y ) \coloneqq \mathds{1}_{Y \setminus
    \Sigma_1}(y) \e_3$ is the solution of the cell problem
  \eqref{eq:Analysis-Cell-problem-for-H-field} with $\oint H^3 =
  \e_3$. The missing entries of the effective permeability matrix
  $\mu_{\eff}$ can now be computed using the formula for $H^3$ and the
  definition of $\mu_{\eff}$ in
  \eqref{eq:Analysis-definition-of-the-effective-permeability}.
\end{proof}

Besides $\hat{\mu}$, we also need the effective permittivity
$\hat{\varepsilon}$.
For $l \in \set{1,2,3}$ we denote by $E^l \in L^2(Y ; \C^3)$ the weak
periodic solution to
\begin{subequations}\label{eq:Analysis-cell-problem-for-electric-field}
  \begin{empheq}[left=\empheqlbrace]{alignat=2}
    \label{eq:3}
    \curl E^l & = 0 && \text{ in }  Y \, ,\\
    \Div E^l & = 0 && \text{ in } Y \setminus \overline{\Sigma}_1\, ,\\
    E^l &= 0 && \text{ in } \Sigma_1\, ,
  \end{empheq}
  with
  \begin{equation}
    \int_Y E^l = \e_l\, .
  \end{equation}
\end{subequations}
Problem~\eqref{eq:Analysis-cell-problem-for-electric-field} is
uniquely solvable by~\cite[Lemma 3.1]{SU17hommaxwell}.
Consequently, the solutions
to~\eqref{eq:Analysis-cell-problem-for-electric-field} are real vector
fields. Indeed, for each index $l \in \set{1,2,3}$ the vector field
$\Im (E^l) \colon Y \to \R^3$ is a weak solution
to~\eqref{eq:Analysis-cell-problem-for-electric-field} with $\int_Y
\Im (E^l) = 0$ and hence $\Im(E^l) = 0$ in $Y$.

As in~\cite{SU17hommaxwell} we set
\begin{equation}\label{eq:Analysis-effective-permittivity}
  \hat{\varepsilon} (x) \coloneqq \varepsilon_{\eff}
  \indicator_{Q_M}(x) + \Id \indicator_{G \setminus
    \overline{Q_M}}(x)\, ,
\end{equation}
where
\begin{equation}\label{eq:definition-effective-permittivity}
  (\varepsilon_{\eff})_{kl} \coloneqq \int_Y E^k \cdot E^l\, .
\end{equation}
\begin{lemma}[Effective permittivity for the metal cylinder]
  For the microstructure $\Sigma = \Sigma_1$, the  permittivity
  $\varepsilon_{\eff}$ is given by
  \begin{equation}\label{eq:Analysis-metal-cylinder-effective-permittivity}
    \varepsilon_{\eff} = \diag(\gamma, \gamma, 0)
\, ,
\end{equation}
where  $\gamma \coloneqq ({\varepsilon}_{\eff})_{1,1}$.
\end{lemma}
\begin{proof}
  As shown in Table~\ref{tab:analysissummary}, we find that
  $\mathcal{N}_{\Sigma_1}= \set{1,2}$. From~\cite[Lemma
  3.2]{SU17hommaxwell} we hence deduce that
  $(\varepsilon_{\eff})_{k,3}$ as well as $(\varepsilon_{\eff})_{3,k}$
  vanish for all $k \in \set{1,2,3}$.  We claim that the matrix
  $\varepsilon_{\eff}$ is symmetric. Because of
  $(\varepsilon_{\eff})_{k,3} = (\varepsilon_{\eff})_{3,k} = 0$ we
  only have to  prove that
  $(\varepsilon_{\eff})_{1,2} = (\varepsilon_{\eff})_{2,1}$. As the
  solutions $E^1$ and $E^2$ of the cell
  problem~\eqref{eq:Analysis-cell-problem-for-electric-field} are real
  vector fields, we compute that
  \begin{equation*}
    (\varepsilon_{\eff})_{1,2} = \int_Y E^1 \cdot E^2 = \int_Y E^2 \cdot
    E^1 = (\varepsilon_{\eff})_{2,1}\, .
  \end{equation*}

 To show that $(\varepsilon_{\eff})_{1,2} =
  (\varepsilon_{\eff})_{2,1} = 0$, we consider the map $M \colon Y \to
  Y$ that is defined by the diagonal matrix $\diag (-1, 1, 1)$.
  Note that $M(\Sigma_1) = \Sigma_1$. To shorten the notation, we set
  $E \coloneqq E^1$. Consider the vector field $F \colon Y \to \R^3$,
  \begin{equation*}
    F(x) \coloneqq ME(Mx) =
    \begin{pmatrix}
      -E_1 \\
      E_2 \\
      E_3
    \end{pmatrix}
(-x_1, x_2, x_3)\, .
\end{equation*}
One readily checks that $F$ is a solution to the cell
problem~\eqref{eq:Analysis-cell-problem-for-electric-field} with
$
  \int_Y F = -\e_1\, .
$
  Due to the unique solvability of the cell
  problem~\eqref{eq:Analysis-cell-problem-for-electric-field}, we conclude that
  that $F=-E $. Similarly, we find that $ME^2\circ M = E^2$. Thus
  \begin{equation*}
  (\varepsilon_{\eff})_{1,2} = \int_Y E^1 \cdot E^2 = -\int_{Y}
  ME^1(My)
    \cdot ME^2(My) \d y = - \int_Y E^1 \cdot E^2 = -
    (\varepsilon_{\eff})_{1,2}\, .
  \end{equation*}
  Hence $(\varepsilon_{\eff})_{1,2} = (\varepsilon_{\eff})_{2,1} = 0$.

  We are left to prove $(\varepsilon_{\eff})_{2,2}
  =(\varepsilon_{\eff})_{1,1}$. To do so, we consider the rotation map $R
  \colon Y \to Y$ which is defined by the matrix
  \begin{equation*}
    \begin{pmatrix}
      0 & -1 & 0\\
      1 & 0 & 0 \\
      0 & 0 & 1
    \end{pmatrix}
    \, .
  \end{equation*}
  Then $R(\Sigma_1) =
  \Sigma_1$. Moreover, as the cell
  problem~\eqref{eq:Analysis-cell-problem-for-electric-field} is uniquely
  solvable, we find that $RE^2\circ R = -E^1$. Thus
  \begin{equation*}
    (\varepsilon_{\eff})_{1,1} = \int_Y E^1 \cdot E^1 = \int_{Y} RE^2(Ry)
    \cdot RE^2(Ry) \d y = \int_Y E^2 \cdot E^2 = (\varepsilon_{\eff})_{2,2}\, .
  \end{equation*}
  This proves the claim.
\end{proof}

By Theorem 4.1 of~\cite{SU17hommaxwell}, the microstructure
$\Sigma_1$ together with the effective
permittivity from~\eqref{eq:Analysis-metal-cylinder-effective-permittivity}
and
permeability from~\eqref{eq:Analysis-metal-cylinder-effective-permeability}
implies that the effective equations are
\begin{subequations}\label{eq:Analysis-effective-equations-metal-cylinder-case-1}
  \begin{empheq}[left=\empheqlbrace]{alignat=2}
    \partial_2 \Heff_3 - \partial_3 \Heff_2 & = - i \omega
    \varepsilon_0
    (\hat{\varepsilon}\Eeff)_1 && \text{ in } G\, ,\\
    \partial_3 \Heff_1 - \partial_1 \Heff_3 & = - i \omega
    \varepsilon_0 (\hat{\varepsilon }\Eeff)_2 && \text{ in } G \label{eq:Analysis-effective-equations-metal-cylinder-perp-case-curlH=E-for-some-components} \, , \\
    \Eeff_3 & = 0 && \text{ in } Q_M \label{eq:Analysis-effective-equations-metal-cylinder-perp-case-E=0-in-QM}\, .
  \end{empheq}
\end{subequations}
The
equations~\eqref{eq:Analysis-effective-equations-metal-cylinder-case-1}
do not repeat~\eqref{eq:Analysis-the-general-effective-equations-1}
and~\eqref{eq:Analysis-the-general-effective-equations-2}.
Due
to~\eqref{eq:Analysis-the-general-effective-equations-1},
the effective electric field is divergence-free. As we assume that $\Eeff$ travels along
the $x_1$-axis, the first component $\Eeff_1$ vanishes. Due
to~\eqref{eq:Analysis-effective-equations-metal-cylinder-perp-case-E=0-in-QM}
we expect no transmission if the effective electric field is polarized
in $\e_3$-direction. For nontrivial transmission, we may therefore make the following ansatz for
the effective electric field $\Eeff \colon G \to \mathbb{C}^3$,
\begin{equation*}
  \Eeff (x) \coloneqq \big( \e^{-ik_0 x_1} + R \e^{i k_0 x_1}\big)
  \e_2 \quad \text{ for } x = (x_1, x_2, x_3) \in Q_R\,.
\end{equation*}
Thanks
to~\eqref{eq:Analysis-the-general-effective-equations-2}
the magnetic field $\Heff$ is given by
\begin{equation*}
  \Heff (x) = - \frac{k_0}{\omega \mu_0} \big( \e^{-i k_0 x_1} - R
  \e^{ i k_0 x_1} \big)\e_3 \quad \text{ for } x = (x_1, x_2, x_3)
  \in Q_R\, .
\end{equation*}
In the meta-material $Q_M$ we write
\begin{equation*}
  \Eeff(x) = \big( T_M \e^{- i k_1 x_1} + R_M \e^{i k_1 x_1}\big)\e_2
  \text{ and } \Heff(x) = - \frac{k_1}{\omega \mu_0 \alpha}\big( T_M \e^{-i k_1 x_1} - R_M
  \e^{i k_1 x_1}) \e_3
\end{equation*}
for $x \in Q_M$,
where we used
equation~\eqref{eq:Analysis-the-general-effective-equations-1}
and~\eqref{eq:Analysis-metal-cylinder-effective-permeability} to
determine the magnetic field with
$\alpha \coloneqq \abs{ Y \setminus \Sigma_1}$. To compute the value
of $k_0$ we use
equation~\eqref{eq:Analysis-the-general-effective-equations-2}
and we find that $k_0 = \omega \sqrt{\varepsilon_0
  \mu_0}$. From~\eqref{eq:Analysis-effective-equations-metal-cylinder-perp-case-curlH=E-for-some-components}
we deduce that $k_1 = k_0 \sqrt{\alpha \gamma}$.  In $Q_L$ we
choose~\eqref{eq:Analysis-general-ansatz-for-E-and-H-in-QL} as the
ansatz for $\Eeff$ and $\Heff$, where $k = 2$ and $l=3$.

\begin{lemma}[Transmission and reflection coefficients]\label{lem:Analysis-metal-cylinder-case-1-transmission-and-reflection-coefficients}
  Given the electric and magnetic fields $\Eeff$ and $\Heff$ as
  described above. Set $\alpha \coloneqq \abs{Y \setminus \Sigma_1}$,
  $k_1 = \omega \sqrt{\varepsilon_0 \mu_0 \alpha \gamma}$, and
  $p_1 \coloneqq \e^{ i k_1 L }$.
   The coefficients are then given by
  \begin{alignat*}{2}
    R &= \frac{( \alpha - \gamma)
    (1 - p_1^2)}{ (\alpha + \gamma)(1-p_1^2)+ 2 \sqrt{\alpha
  \gamma} (1+ p_1^2)}\, , \quad
    T_M & =\frac{2 \sqrt{\alpha} ( \sqrt{\alpha} +
  \sqrt{\gamma})}{(\alpha + \gamma)(1-p_1^2)+ 2 \sqrt{\alpha
  \gamma} (1+ p_1^2)}\, ,
    \\
    R_M &= - \frac{2 \sqrt{\alpha} p_1^2 ( \sqrt{\alpha} -
   \sqrt{\gamma})}{ (\alpha + \gamma)(1-p_1^2)+ 2 \sqrt{\alpha
  \gamma} (1+ p_1^2)} \, , \quad
    T & =  \frac{4 \sqrt{ \alpha\gamma}p_1}{ (\alpha + \gamma)(1-p_1^2)+ 2 \sqrt{\alpha
  \gamma} (1+ p_1^2)} \, .
  \end{alignat*}
\end{lemma}

\begin{proof}
  By~\eqref{eq:Analysis-the-general-effective-equations-1}
  the tangential trace of $\Eeff$ has no jump across the surfaces
  $\set{ x \in G \given x_1 = 0}$ and $\set{ x \in G \given x _1 = -
    L}$. Thus
  \begin{equation} \label{eq:Analysis-equation-for-transmission-and-reflection-metal-cylinder-case-1-1}
  T_M + R_M = 1+R \quad \text{ and } \quad T = p_1 T_M + \frac 1p_1
 R_M\, .
\end{equation}
The effective field $\Heff$ is parallel to $\e_3$ and hence,
by~\eqref{eq:Analysis-effective-equations-metal-cylinder-perp-case-curlH=E-for-some-components},
the third component $\Heff_3$ does not jump across the surfaces $\set{
x \in G \given x_1 = 0}$ and $\set{ x \in G \given x_1 = -L}$. We
may therefore conclude that
\begin{equation}\label{eq:Analysis-equation-for-transmission-and-reflection-metal-cylinder-case-1-2}
  \sqrt{\frac{\gamma}{\alpha}} \big(T_M-R_M \big) = 1- R \quad \text{ and } \quad
  T = \sqrt{\frac{\gamma}{\alpha}} \bigg( p_1 T_M - \frac 1p_1 R_M
  \bigg)\, .
\end{equation}
Here we used that $k_0 = \omega \sqrt{ \mu_0 \varepsilon_0}$ and $k_1
= k_0 \sqrt{ \alpha \gamma}$.
Solving the equations on the left-hand side
in~\eqref{eq:Analysis-equation-for-transmission-and-reflection-metal-cylinder-case-1-1}
and~\eqref{eq:Analysis-equation-for-transmission-and-reflection-metal-cylinder-case-1-2}
for $R$ and the other two equations for $T$,
we find that
\begin{align}\label{eq:Analysis-equation-for-transmission-and-reflection-metal-cylinder-case-1-3}
  T_M + R_M - 1& = R = 1- \sqrt{\frac{\gamma}{\alpha}}(T_M-R_M)
  \shortintertext{and }
  p_1T_M + \frac{1}{p_1}R_M& = T = \sqrt{\frac{\gamma}{\alpha}}
  \bigg(p_1 T_M - \frac{1}{p_1}R_M \bigg)\, .\label{eq:Analysis-equation-for-transmission-and-reflection-metal-cylinder-case-1-5}
\end{align}
Setting $d_+ \coloneqq 1+ \sqrt{\gamma/\alpha}$ and
$d_{-}\coloneqq 1 - \sqrt{\gamma/\alpha}$,
equations~\eqref{eq:Analysis-equation-for-transmission-and-reflection-metal-cylinder-case-1-3}
and~\eqref{eq:Analysis-equation-for-transmission-and-reflection-metal-cylinder-case-1-5}
can be written as
\begin{equation}\label{eq:eq:Analysis-equation-for-transmission-and-reflection-metal-cylinder-case-1-4}
  d_+ T_M = 2 - d_{-} R_M \quad \text{ and } \quad
  p_1 d_{-}T_M = - \frac{1}{p_1} d_+ R_M\, .
\end{equation}
Solving each of the two equations
in~\eqref{eq:eq:Analysis-equation-for-transmission-and-reflection-metal-cylinder-case-1-4}
for $R_M$ and then equating the two
expressions for $R_M$, we obtain
\begin{align*}\label{eq:Analysis-equation-for-transmission-and-reflection-metal-cylinder-case-1-5}
  T_M &= \frac{2 d_+}{d_+^2 - d_{-}^2p_1^2}
      = \frac{ 2 ( 1 + \sqrt{\gamma / \alpha}) }{
         (1 + \sqrt{\gamma / \alpha} )^2
      - (1 - \sqrt{\gamma / \alpha}
  )^2p_1^2}
   = \frac{2 \sqrt{\alpha} ( \sqrt{\alpha} +
  \sqrt{\gamma})}{(\sqrt{\alpha} + \sqrt{\gamma})^2 -
  (\sqrt{\alpha} - \sqrt{\gamma})^2 p_1^2}\, .
\end{align*}
Note that $(\sqrt{\alpha} + \sqrt{\gamma})^2 - (\sqrt{\alpha} -
\sqrt{\gamma})^2p_1^2 = (\alpha + \gamma)(1-p_1^2)+ 2 \sqrt{\alpha
  \gamma} (1+ p_1^2)$, which yields the formula for $T_M$.
From the second equation
in~\eqref{eq:eq:Analysis-equation-for-transmission-and-reflection-metal-cylinder-case-1-4},
we deduce that
\begin{equation*}
  R_M = - p_1^2\frac{d_{-}}{d_+}T_M
  = - p_1^2 \frac{\sqrt{\alpha} - \sqrt{\gamma}}{\sqrt{\alpha} +
    \sqrt{\gamma}} T_M
  = - \frac{2 \sqrt{\alpha} p_1^2 ( \sqrt{\alpha} -
   \sqrt{\gamma})}{ (\alpha + \gamma)(1-p_1^2)+ 2 \sqrt{\alpha
  \gamma} (1+ p_1^2)}\, .
\end{equation*}

By~\eqref{eq:Analysis-equation-for-transmission-and-reflection-metal-cylinder-case-1-3},
we have that
\begin{align*}
  R &= T_M + R_M - 1 = \frac{2 \sqrt{\alpha} (\sqrt{\alpha} +
    \sqrt{\gamma} ) - 2 \sqrt{\alpha} p_1^2 (\sqrt{\alpha} - \sqrt{\gamma} )}{ (\alpha + \gamma)(1-p_1^2)+ 2 \sqrt{\alpha
  \gamma} (1+ p_1^2)} - 1 \\
  &= \frac{( \alpha - \gamma)
    (1 - p_1^2)}{ (\alpha + \gamma)(1-p_1^2)+ 2 \sqrt{\alpha
  \gamma} (1+ p_1^2)}\, .
\end{align*}
To compute the coefficient $T$ we use
equation~\eqref{eq:Analysis-equation-for-transmission-and-reflection-metal-cylinder-case-1-1}
and find that
\begin{align*}
  T &= \frac{2 \sqrt{\alpha} ( \sqrt{\alpha} +
  \sqrt{\gamma})p_1 - 2 \sqrt{\alpha} (\sqrt{\alpha} - \sqrt{\gamma})p_1}{ (\alpha + \gamma)(1-p_1^2)+ 2 \sqrt{\alpha
  \gamma} (1+ p_1^2)}
  = \frac{4 \sqrt{ \alpha} \sqrt{\gamma}p_1}{ (\alpha + \gamma)(1-p_1^2)+ 2 \sqrt{\alpha
  \gamma} (1+ p_1^2)}\, .
\end{align*}
This proves the claim.
\end{proof}

\subsubsection{The metal cylinder $\Sigma_2$}
\label{sec:metal-cylinder-case}
Similar to the previous section, we shall determine the transmission
and reflection coefficients for a metal cylinder, considering the
microstructure $\Sigma_2$. We define the effective permeability and
the effective permittivity
$\hat{\mu}, \hat{\varepsilon } \colon G \to \C^3$ as
in~\eqref{eq:Analysis-effective-permeability}
and~\eqref{eq:Analysis-effective-permittivity}.
Following the
reasoning of Section~\ref{sec:metal-cylinder}, we find that the
$\mu_{\eff}$ and
$\varepsilon_{\eff}$ are given by
\begin{equation*}
\mu_{\eff} = \diag \big( \abs{Y \setminus \Sigma_2}, 1, 1 \big)
  \quad \text{ and } \quad
  \varepsilon_{\eff} = \diag (0, \gamma, \gamma)\, ,
\end{equation*}
where $\gamma \in \mathbb{C}$ is defined as $\gamma \coloneqq \int_Y E^2
\cdot E^2$. The effective equations for the  microstructure
$\Sigma_2$ are
\begin{subequations}
  \begin{empheq}[left=\empheqlbrace]{alignat=2}
    \partial_3 \Heff_1 - \partial_1 \Heff_3 & = -i \omega
    \varepsilon_0 (\hat{\varepsilon} \Eeff)_2 && \text{ in } G \, , \label{eq:Analysis-effective-equations-metal-cylinder-case-2-curlH=E-for-some-components} \\
    \partial_1 \Heff_2 - \partial_2 \Heff_1 & = - i \omega
    \varepsilon_0 (\hat{\varepsilon } \Eeff)_3 && \text{ in } G \, ,
    \\
    \Eeff_1& = 0 && \text{ in } Q_M\, .
  \end{empheq}
\end{subequations}

We may take a similar ansatz for the effective fields as in
Section~\ref{sec:metal-cylinder} and obtain the following transmission
and reflection coefficients. Note that $k_1$ in
Section~\ref{sec:metal-cylinder} has to be replaced by $k_2 \coloneqq
k_0 \sqrt{\gamma}$.
\begin{lemma}[Transmission and reflection coefficients]
 Within the setting of Section~\ref{sec:metal-cylinder}, we set $k_2 = \omega \sqrt{\varepsilon_0
    \mu_0 \gamma}$ and $p_2 \coloneqq e^{ i k_2 L}$. The reflection and transmission
  coefficients for $\Sigma_2$ are given by
  \begin{alignat*}{3}
    R &= \frac{(1- \gamma) (1-p_2^2)}{ (1+ \gamma)(1-p_2^2) + 2
      \sqrt{\gamma}(1+p_2^2)}\, , & \quad & T_M &=\frac{2 (1 +
      \sqrt{\gamma})}{(1+
      \gamma)(1-p_2^2) + 2 \sqrt{\gamma}(1+p_2^2)}\, , \\
    R_M &= - \frac{2 p_2^2(1 - \sqrt{\gamma})}{ (1+ \gamma)(1-p_2^2) +
      2 \sqrt{\gamma}(1+p_2^2)}\, , &\quad & T &= \frac{4 p_2
      \sqrt{\gamma}}{ (1+ \gamma)(1-p_2^2) + 2
      \sqrt{\gamma}(1+p_2^2)} \, .
  \end{alignat*}
\end{lemma}
Note that in the above transmission and reflection coefficients the
volume fraction of air $\alpha = \abs{ Y \setminus \Sigma_2}$ does not
appear. This is different for the metal cylinder $\Sigma_1$; see Lemma~\ref{lem:Analysis-metal-cylinder-case-1-transmission-and-reflection-coefficients}.

\begin{proof}
Thanks
to~\eqref{eq:Analysis-the-general-effective-equations-1}
we know that the tangential components of $\Eeff$ do not jump across
the surfaces $\set{ x \in G \given x_1 = 0}$ and $\set{x \in G
  \given x_1 = -L}$. Hence
\begin{equation}\label{eq:Analysis-equations-for-transmission-and-reflection-coefficients-metal-plate-case-2-1}
1+R = T_M + R_M \quad \text{ and } T = p_2 T_M + \frac{1}{p_2}R_M \, .
\end{equation}
The effective field $\Heff$ is parallel to $\e_3$ and hence, due
to~\eqref{eq:Analysis-effective-equations-metal-cylinder-case-2-curlH=E-for-some-components},
the third component $\Heff_3$ does neither jump across $\set{ x \in G
  \given x_1 = 0}$ nor across $\set{ x \in G \given x_1 = -L}$. We may
therefore conclude that
\begin{equation}\label{eq:eq:Analysis-equations-for-transmission-and-reflection-coefficients-metal-plate-case-2-2}
  1- R = \sqrt{\gamma} (T_M - R_M) \quad \text{ and } \quad T =
  \sqrt{\gamma} \bigg(p_2 T_M - \frac{1}{p_2} R_M \bigg)\, .
\end{equation}
Here we used that $k_2 = k_0 \sqrt{\gamma} = \omega \sqrt{
  \varepsilon_0 \mu_0 \gamma}$.

Solving the equations on the left-hand side
in~\eqref{eq:Analysis-equations-for-transmission-and-reflection-coefficients-metal-plate-case-2-1}
and~\eqref{eq:eq:Analysis-equations-for-transmission-and-reflection-coefficients-metal-plate-case-2-2}
for $R$ and the other two for $T$, we find that
\begin{equation*}
  T_M +R_M - 1 =R = 1 - \sqrt{\gamma} (T_M - R_M) \: \text{ and }
  \:
   p_2 T_M + \frac{1}{p_2} R_M =T = \sqrt{\gamma} \bigg(p_2 T_M - \frac{1}{p_2}R_M \bigg)\, .
\end{equation*}
Setting $c_+ \coloneqq 1 + \sqrt{\gamma}$ and $c_{-} \coloneqq 1 -
\sqrt{\gamma}$, these two equations can be re-written as
\begin{equation}\label{eq:eq:Analysis-equations-for-transmission-and-reflection-coefficients-metal-plate-case-2-3}
  c_+ T_M = 2 - c_{-}R_M \quad \text{ and } \quad  c_{-}p_2 T_M = -
  \frac{1}{p_2} c_+ R_M\, .
\end{equation}
We can solve for $T_M$ and obtain
\begin{equation*}
  T_M = \frac{2 c_+}{c_+^2 - c_{-}^2p_2^2} = \frac{2 (1 +
    \sqrt{\gamma})}{(1 + \sqrt{\gamma})^2 - (1 - \sqrt{\gamma})^2
    p_2^2}\, .
\end{equation*}
Note that $(1+ \sqrt{\gamma})^2 - (1- \sqrt{\gamma})^2p_2^2 = (1+
\gamma)(1-p_2^2) + 2 \sqrt{\gamma}(1+p_2^2)$, which proves the formula
for $T_M$.
By~\eqref{eq:eq:Analysis-equations-for-transmission-and-reflection-coefficients-metal-plate-case-2-3}
we then conclude that
\begin{equation*}
  R_M = - p_2^2\frac{c_{-}}{c_+}T_M = - p_2^2 \frac{1 -
    \sqrt{\gamma}}{1+\sqrt{\gamma}}T_M = - \frac{2 p_2^2(1 - \sqrt{\gamma})}{ (1+
\gamma)(1-p_2^2) + 2 \sqrt{\gamma}(1+p_2^2)}\, .
\end{equation*}
To determine the coefficient $R$ we recall from above that $R = T_M +
R_M - 1$ and hence
\begin{equation*}
  R =  \frac{2 (1+ \sqrt{\gamma}) - 2 p_2^2 (1 -
    \sqrt{\gamma}) - (1 + \sqrt{\gamma})^2 + (1 - \sqrt{\gamma})^2p_2^2}{ (1+
\gamma)(1-p_2^2) + 2 \sqrt{\gamma}(1+p_2^2)}
  = \frac{(1- \gamma) (1-p_2^2)}{ (1+
\gamma)(1-p_2^2) + 2 \sqrt{\gamma}(1+p_2^2)}\, .
\end{equation*}
As $T = p_2 T_M + 1/p_2 R_M$, we find that
\begin{equation*}
  T = \frac{2 p_2 (1+ \sqrt{\gamma}) - 2 p_2 (1- \sqrt{\gamma})}{ (1+
\gamma)(1-p_2^2) + 2 \sqrt{\gamma}(1+p_2^2)} = \frac{4 p_2 \sqrt{\gamma}}{ (1+
\gamma)(1-p_2^2) + 2 \sqrt{\gamma}(1+p_2^2)}\, .
\end{equation*}
This proves the claim.
\end{proof}

We chose the same polarization for the electric and the magnetic field
as in Section~\ref{sec:metal-cylinder}. By symmetry of the
microstructure,  we may as well assume that $\Eeff$ is parallel to $\e_3$ and
$\Heff$ is parallel to $\e_2$ and obtain the same reflection and
transmission coefficients.

\subsubsection{The metal plate}
\label{sec:metal-plate}
We consider the microstructure $\Sigma_3$; that is, a metal plate
which is perpendicular to $\e_2$.
Following the reasoning in Section 5.2 in~\cite{SU17hommaxwell}, we
determine the effective equations and obtain:
\begin{subequations}
  \begin{empheq}[left=\empheqlbrace]{alignat=2}
    \partial_3 \hat{H}_1 - \partial_1 \hat{H}_3 & = - i \omega
    \varepsilon_0
    \alpha^{-1} \hat{E}_2 && \text{ in } Q_M\, , \label{eq:Analysis-effective-equations-metal-plate-1-case-curlH=E-for-some-components} \\
    \hat{E}_1 &= \hat{E}_3  = 0 && \text{ in } Q_M\, , \label{eq:Analysis-effective-equations-metal-plate-1-case-electric-field-in-meta-material}\\
    \hat{H}_2 &= 0 && \text{ in } Q_M\, ,
  \end{empheq}
\end{subequations}
where $\alpha \coloneqq \abs{Y \setminus \Sigma_3}$.

The electromagnetic wave is assumed to travel in $\e_1$-direction from right to
left. Moreover, by~\eqref{eq:Analysis-the-general-effective-equations-1},
the electric field is divergence free. Hence, the first component
$\Eeff_1$ vanishes. Because of~\eqref{eq:Analysis-effective-equations-metal-plate-1-case-electric-field-in-meta-material}
we expect no transmission if the electric field is polarized in
$\e_3$-direction. We may therefore make the following ansatz for the
effective electric field $\hat{E} \colon G \to \mathbb{C}^3$,
\begin{equation*}
  \hat{E}(x) \coloneqq \big(\e^{-i k_0 x_1} + R \e^{i k_0
    x_1}\big)\e_2\quad \text{ for } x = (x_1, x_2, x_3) \in Q_R\, .
\end{equation*}
Thanks to~\eqref{eq:Analysis-the-general-effective-equations-2},
the magnetic field $\hat{H}$ is given by
\begin{equation*}
  \hat{H}(x) =- \frac{k_0}{\omega \mu_0} \big( \e^{-i k_0 x_1} - R
  \e^{i k_0 x_1} \big) \e_3 \quad \text{ for } x \in Q_R\, .
\end{equation*}
By
equation~\eqref{eq:Analysis-effective-equations-metal-plate-1-case-electric-field-in-meta-material},
the first and the third component of
the effective electric field are trivial; from this and
equation~\eqref{eq:Analysis-the-general-effective-equations-1},
we deduce that
\begin{equation*}
  \hat{E}(x) = \big(T_M \e^{-i k_3 x_1} + R_M\e^{i k_3 x_1}\big) \e_2
  \quad \text{ and } \quad \hat{H}(x) = -\frac{k_3}{\omega \mu_0
    \alpha} \big(T_M \e^{-ik_3x_1} - R_M \e^{i k_3x_1} \big)\e_3\: \text{ in
  } Q_M\, .
\end{equation*}
The value of $k_3$ can be determined
by~\eqref{eq:Analysis-effective-equations-metal-plate-1-case-curlH=E-for-some-components}
and we find that $k_3 = k_0 = \omega \sqrt{\varepsilon_0 \mu_0}$. In
$Q_L$, we choose~\eqref{eq:Analysis-general-ansatz-for-E-and-H-in-QL}
as the ansatz for $\hat{E}$ and $\hat{H}$, where $k = 2$ and $l = 3$.

\begin{lemma}[Transmission and reflection coefficients]
  Given the effective fields $\hat{E}$ and $\hat{H}$ as described
  above. Set $\alpha \coloneqq \abs{Y \setminus \Sigma_3}$ and $p_0
  \coloneqq \e^{ i \omega \sqrt{\varepsilon_0 \mu_0} L}$. The
  reflection and transmission coefficients are given by
  \begin{alignat}{3}
    R &=  \frac{(\alpha^2 - 1)(1- p_0^2)}{(1 + \alpha^2)(1-p_0^2) + 2
        \alpha (1+ p_0^2)}\, , &\qquad & T_M
    &= \frac{2 \alpha (\alpha +1)}{(1 + \alpha^2)(1-p_0^2) + 2
        \alpha (1+ p_0^2)}\,
      ,
      \\
      R_M &= -\frac{2 \alpha p_0^2 (\alpha - 1)}{(1 + \alpha^2)(1-p_0^2) + 2
        \alpha (1+ p_0^2)} \, ,& \qquad & T & = \frac{4 p_0 \alpha}{(1 + \alpha^2)(1-p_0^2) + 2
        \alpha (1+ p_0^2)} \, .
  \end{alignat}
\end{lemma}

\begin{proof}
From~\eqref{eq:Analysis-the-general-effective-equations-1}
we deduce that $\curl \hat{E}$ has no singular part and hence the
tangential trace of $\hat{E}$ along the surfaces $\set{x \in G \given
  x_1 = 0}$ and $\set{x \in G \given x_1 = -L}$ does not jump. Thus
\begin{equation}\label{eq:Analysis-equations-for-transmission-and-reflection-coefficients-metal-plate-1}
   T_M + R_M= 1+R \quad \text{ and } \quad T = p_0T_M + \frac{1}{p_0} R_M\,.
\end{equation}
As $\hat{H}$ is parallel to $\e_3$, we deduce
from~\eqref{eq:Analysis-effective-equations-metal-plate-1-case-curlH=E-for-some-components}
that $\hat{H}_3$ does not jump across the surfaces $\set{x \in G
  \given x_1 = 0}$ and $\set{x \in G \given x_1 = -L}$. This implies
that
\begin{equation}\label{eq:Analysis-equations-for-transmission-and-reflection-coefficients-metal-plate-2}
  1 - R = \frac{1}{\alpha} \big(T_M - R_M \big) \quad \text{ and }
  \quad T = \frac{1}{\alpha}\Big( p_0 T_M - \frac{1}{p_0} R_M\Big) \, .
\end{equation}
Here we used that $k_0 = k_3 = \omega \sqrt{\varepsilon_0 \mu_0}$.
Note that $\alpha >0$ and hence we find $a >0$  such that $\sqrt{a} =
1/ \alpha$. With this new parameter $a$, the equations
in~\eqref{eq:Analysis-equations-for-transmission-and-reflection-coefficients-metal-plate-2}
read
\begin{equation}\label{eq:Analysis-equations-for-transmission-and-reflection-coefficients-metal-plate-3}
  \sqrt{a}(T_M-R_M) = 1-R \quad \text{ and } \quad
  T = \sqrt{a} \bigg(p_0 T_M - \frac{1}{p_0}R_M \bigg) \, .
\end{equation}
Thus the equations
in~\eqref{eq:Analysis-equations-for-transmission-and-reflection-coefficients-metal-plate-1}
and~\eqref{eq:Analysis-equations-for-transmission-and-reflection-coefficients-metal-plate-3}
have the same structure as the equations
in~\eqref{eq:Analysis-equations-for-transmission-and-reflection-coefficients-metal-plate-case-2-1}
and~\eqref{eq:eq:Analysis-equations-for-transmission-and-reflection-coefficients-metal-plate-case-2-2}. We
may therefore use the formulas for $R, T, R_M$, and $T_M$ derived in
Section~\ref{sec:metal-cylinder-case}. Note that
\begin{equation*}
  1+ \sqrt{a} = \frac{\alpha +1}{\alpha}\,  \quad 1- \sqrt{a} =
  \frac{\alpha - 1}{\alpha}\, , \quad \text{ and } \quad 1-a =
  \frac{\alpha^2 - 1}{\alpha^2}\, .
\end{equation*}
Thus
\begin{align*}
  T_M &= \frac{2 (1+ \sqrt{a})}{(1+ \sqrt{a})^2 - (1- \sqrt{a})^2p_0^2}
        = \frac{2 \alpha (\alpha +1)}{(\alpha + 1)^2 -
        (\alpha-1)^2p_0^2}
        = \frac{2 \alpha (\alpha +1)}{(1 + \alpha^2)(1-p_0^2) + 2
        \alpha (1+ p_0^2)}\,
        , \\
  R_M &= - \frac{2 p_0^2 (1- \sqrt{a})}{(1+ \sqrt{a}) - (1-
        \sqrt{a})p_0^2}
        = - \frac{2 \alpha p_0^2 (\alpha - 1)}{(\alpha + 1)^2 - (\alpha
        -1)^2p_0^2}
        = -\frac{2 \alpha p_0^2 (\alpha - 1)}{(1 + \alpha^2)(1-p_0^2) + 2
        \alpha (1+ p_0^2)}\, ,\\
  R &= \frac{(1-a)(1-p_0^2)}{(1 + \sqrt{a}) - (1- \sqrt{a})p_0^2}
      = \frac{(\alpha^2 - 1)(1-p_0^2)}{(\alpha +1)^2 -
      (\alpha-1)^2p_0^2}
      = \frac{(\alpha^2 - 1)(1- p_0^2)}{(1 + \alpha^2)(1-p_0^2) + 2
        \alpha (1+ p_0^2)}\, ,
      \shortintertext{and}
      T &= \frac{4 p_0\sqrt{\alpha}}{(1+ \sqrt{\alpha})^2 - (1-
          \sqrt{\alpha})^2p_0^2}
          = \frac{4 p_0 \alpha}{(\alpha+1)^2 - (\alpha-1)^2 p_0^2}
          = \frac{4 p_0 \alpha}{(1 + \alpha^2)(1-p_0^2) + 2
        \alpha (1+ p_0^2)}\, .
\end{align*}
This proves the claim.
\end{proof}

\subsubsection{The air cylinder}
\label{sec:air-cylinder}
We consider the microstructure $\Sigma_4$; that is, an air cylinder
with symmetry axis parallel to $\e_1$ (see Fig.~\ref{fig:the air
  cylinder}).
Combining the effective
equations~\eqref{eq:Analysis-the-general-effective-equations} with the
index sets in Table~\ref{tab:analysissummary}, we obtain the effective
system for this case:
\begin{subequations}
  \begin{empheq}[left=\empheqlbrace]{alignat=2}
    \Eeff & = 0 && \text{ in } Q_M \,
    , \label{eq:Analysis-effective-equations-air-cylinder-E=0}\\
    \Heff_2 & = \Heff_3  = 0 && \text{ in } Q_M \, . \label{eq:Analysis-effective-equations-air-cylinder-H=0-for-some-components}
  \end{empheq}
\end{subequations}
As in the previous sections, we choose the following ansatz for the
effective fields $\Eeff, \Heff \colon G \to \mathbb{C}^3$,
\begin{equation*}
  \Eeff(x) \coloneqq \big( \e^{-i k_0 x_1} + R \e^{i k_0 x_1}
  \big)\e_2 \quad \text{ and } \quad \Heff(x) =-  \big( \e^{i k_0 x_1} - R
  \e^{-i k_0 x_1}\big) \e_3 \quad
  \text{ for } x \in  Q_R\, .
\end{equation*}
Equation~\eqref{eq:Analysis-the-general-effective-equations-2}
determines the wave number and we find that $k_0 = \omega \sqrt{\varepsilon_0 \mu_0}$. The effective
electric field $\Eeff$ vanishes in the meta-material $Q_M$ and hence,
by~\eqref{eq:Analysis-the-general-effective-equations-1}
and~\eqref{eq:Analysis-effective-equations-air-cylinder-H=0-for-some-components},
there is also no effective magnetic field in $\Q_M$. So $\Eeff =
\Heff = 0$ in
$Q_M$. Equation~\eqref{eq:Analysis-the-general-effective-equations-1}
implies that the tangential trace of $\Eeff$ does not jump across the
surface $\set{ x \in G \given x_1 =0}$. Thus
\begin{equation*}
   R = -1\, .
\end{equation*}
As no field is transmitted through the meta-material $Q_M$, there is
neither an electric nor an magnetic field in $Q_L$. In other words,
$\Eeff = \Heff = 0$ in $Q_L$. We have thus shown that
\begin{equation*}
   R = -1 \quad \text{ and } \quad T = 0\, .
\end{equation*}

\subsection{Vanishing limiting fields in high-contrast media,
  2D-analysis}
\label{sec:high-contrast-media}

In this section, we perform an analysis of high-contrast media. Of the
four geometries $\Sigma_1$ to $\Sigma_4$, we study the two
$\e_3$-invariant geometries: the metal cylinder $\Sigma_1$ and the
metal plate $\Sigma_3$, compare Fig.~\ref
{fig:Two-examples-of-perfect-conductors-for-analysis}. We analyze the
time-harmonic Maxwell equations \eqref{eq:time-harmonic-Maxwell-eq}
with the high-contrast permittivity $\varepsilon_\eta$ of \eqref
{eq:Analysis-high-contrast-permittivity}.  We recall that the sequence
of solutions $(E^\eta, H^\eta)_\eta$ is assumed to satisfy the
$L^2(G)$-bound \eqref
{eq:Analysis-high-contrast-boundedness-assumption-on-fields}.  We are
interested in the limit behaviour of $(E^{\eta}, H^{\eta})_\eta$ as
$\eta \to 0$.

When we consider perfect conductors, the effective equations \eqref
{eq:Analysis-the-general-effective-equations} imply that some
components of $E^\eta$ or $H^\eta$ converge weakly to $0$ in the
meta-material $Q_M$. For media with high-contrast, we do not have such
a result (we recall that homogenization usually considers compactly
contained geometries $\Sigma\subset Y$).  In this section we ask for
$\Sigma_1$ and $\Sigma_3$: do the electric fields $(E^\eta)_{\eta}$
converge weakly in $L^2(Q_M; \mathbb{C}^3)$ to $0$ as $\eta \to 0$? Is
this weak convergence in fact a strong convergence? The same questions
are considered for the magnetic fields $(H^\eta)_\eta$.

Let us point out that $E^\eta\cdot \mathds{1}_{\Sigma_\eta} \to 0$ in
$L^2(Q_M; \mathbb{C}^3)$. Indeed, the $L^2$-estimate \eqref
{eq:Analysis-high-contrast-boundedness-assumption-on-fields} can be
improved to
\begin{equation}
  \label{eq:Analysis-high-contrast-improved-boundedness-of-fields}
  \sup_{\eta >0} \int_G \Big(\abs{\varepsilon_\eta}\, \abs{E^\eta}^2 +
  \abs{H^\eta}^2\Big) < \infty\, ,
\end{equation}
as was shown in \cite[Section 3.1]{BS10splitring}. Thus
\begin{equation}
  \label{eq:E-eta-in-Sigma-eta}
  \frac{\varepsilon_1}{\eta^2} \int_{\Sigma_\eta} \abs{E^\eta}^2 = \int_G
  \abs{\varepsilon_\eta} \,
  \abs{E^\eta}^2 \mathds{1}_{\Sigma_\eta}
  \leq \int_G \Big(\abs{\varepsilon_\eta}\, \abs{E^\eta}^2 +
  \abs{H^\eta}^2\Big) \leq C\, .
\end{equation}
So we have that $\norm{L^2(G)}{E^\eta \mathds{1}_{\Sigma_\eta}}^2 \leq
\eta^2 C$ which implies that $E^\eta \mathds{1}_{\Sigma_\eta} \to 0$
in $L^2(Q_M; \mathbb{C}^3)$ as $\eta \to 0$.

We recall that the two geometries of interest are
$x_3$-independent. We therefore consider two different cases: In
Section \ref{sec:parall-electr-field}, we study electric fields
$E^\eta$ that are parallel to $\e_3$.  In Section
\ref{sec:parall-magn-field}, we study magnetic fields $H^\eta$ that
are parallel to $\e_3$. By linearity of the equations, superpositions
of these two cases provide the general behaviour of the material.

We will assume that the fields are $x_3$-independent. This is a strong
assumption, which can be justified for $x_3$-independent incoming
fields with a uniqueness property of solutions. In the rest of this section
the fields $E^\eta(x)$ and $H^\eta(x)$ depend only on $(x_1, x_2)$.

\smallskip \textbf{Results for high-contrast media.} In the
$E$-parallel setting, the electric fields $(E^\eta)_{\eta}$ converge
strongly to $0$ in $L^2(Q_M; \mathbb{C}^3)$, the magnetic fields
converge weakly to $0$ in $L^2(Q_M; \mathbb{C}^3)$. On the other hand,
when the magnetic fields $H^\eta$ are parallel to $\e_3$, we can
neither expect the electric fields nor the magnetic fields to converge
weakly to $0$ in $L^2(Q_M; \mathbb{C}^3)$.

\subsubsection{Parallel electric field}
\label{sec:parall-electr-field}

We consider here the case of parallel electric fields, i.e.,
$E^\eta(x) \coloneqq (0, 0, u^\eta(x))$ with $u^\eta = u^\eta(x_1,
x_2)$. By abuse of notation, we will consider $G$ also as a domain in
$\R^2$ and write $(x_1, x_2)\in G$ when $(x_1, x_2,0)\in G$; similarly
for $\Sigma_\eta$. In this setting, the magnetic field $H^\eta$ has no
third component, $H^\eta(x) = (H_1^\eta(x_1, x_2), H_2^\eta(x_1, x_2),
0)$, and Maxwell's equations \eqref{eq:time-harmonic-Maxwell-eq}
reduce to the two-dimensional system
\begin{subequations}
  \label{Analysis-high-contrast-two-dimensional-system-of-Maxwell}
  \begin{empheq}[left=\empheqlbrace]{alignat=2}
    - \nabla^{\perp} u^\eta &
    =\phantom{-}  i \omega \mu_0 (H_1^\eta, H_2^\eta) &&
    \text{ in } G \, , \\
    \nabla^{\perp}\cdot (H^\eta_1, H^\eta_2) & =- i \omega
    \varepsilon_0 \varepsilon_\eta u^\eta && \text{ in } G \, ,
  \end{empheq}
\end{subequations}
where we used the two-dimensional orthogonal gradient,
$\nabla^{\perp}u \coloneqq (- \partial_2 u, \partial_1 u)$, as well as
the two-dimensional curl, $\nabla^{\perp} \cdot (H_1, H_2) \coloneqq
-\partial_2 H_1 +
\partial_1 H_2$. The
system~\eqref{Analysis-high-contrast-two-dimensional-system-of-Maxwell}
can equivalently be written as a scalar
Helmholtz equation
\begin{equation}\label{eq:Analysis-high-contrast-E-parallel-scalar-Helmholtz}
  - \Delta u^\eta = \omega ^2 \varepsilon_0 \mu_0 \varepsilon_\eta u^\eta
  \quad \text{ in } G \subset \R^2\, .
\end{equation}
A solution of this Helmholtz equation provides the fields in the form
$E^\eta = (0 ,0 ,u^\eta)$ and $H^\eta = i (\omega \mu_0)^{-1}
(\nabla^{\perp}u^\eta ,0)$.

\begin{lemma}[Trivial limits for $E^\eta \parallel \e_3$]
  For $\eta > 0$ small, let $\Sigma_\eta \subset G \subset \R^2$ be
  a microscopic geometry that is given by $\Sigma_1$ or $\Sigma_3$, and let
  the permittivity $\varepsilon_\eta \colon G \to \mathbb{C}$ be
  defined by \eqref {eq:Analysis-high-contrast-permittivity}. Let
  $E^\eta, H^\eta \colon G \to \mathbb{C}^3$ be solutions with
  $E^\eta(x) = (0,0,u^\eta(x_1,x_2))$ that satisfy the estimate \eqref
  {eq:Analysis-high-contrast-boundedness-assumption-on-fields}.  Then
  \begin{equation*}
    E^\eta \to 0  \quad \text{ and } \quad H^\eta \weakto 0
    \quad\text{ in } L^2(Q_M) \text{ as } \eta \to 0\, .
  \end{equation*}
\end{lemma}

\begin{proof}
  The $L^2$-boundedness of $E^\eta$ implies the $L^2$-boundedness of
  $u^\eta$, and the $L^2$-boundedness of $H^\eta$ implies the
  $L^2$-boundedness of $\nabla u^\eta$. Therefore, the sequence
  $(u^\eta)_{\eta}$ is bounded in $H^1(G)$, and we find a limit
  function $u \in H^1(G)$ such that $u^\eta \weakto u$ in $H^1(G)$ and
  $u^\eta \to u$ in $L^2(G)$ as $\eta \to 0$.

  We write
  \begin{equation}\label{eq:strong-convergence-in-microstructure}
    u^\eta \mathds{1}_{Q_M}
    = u^\eta \mathds{1}_{Q_M\setminus \Sigma_\eta}
    + u^\eta \mathds{1}_{\Sigma_\eta}\,.
  \end{equation}
  The left hand side converges strongly to $u \mathds{1}_{Q_M}$.  The
  first term on the right hand side of \eqref
  {eq:strong-convergence-in-microstructure} is the product of a
  strongly $L^2(Q_M)$-convergent sequence and a weakly
  $L^2(Q_M)$-convergent sequence: $\mathds{1}_{Q_M \setminus \Sigma_\eta} \weakto
  \alpha$ in $L^2(Q_M)$, where $\alpha \in (0,1)$ is the volume fraction
  of $Y \setminus \Sigma$. We find that the first term on the right hand side
  converges in the sense of distributions to $\alpha u$.  The estimate
  \eqref {eq:E-eta-in-Sigma-eta} provides the strong convergence of
  the second term on the right hand side of \eqref
  {eq:strong-convergence-in-microstructure} to zero.  The
  distributional limit of \eqref
  {eq:strong-convergence-in-microstructure} provides
  \begin{equation}\label{eq:strong-convergence-in-microstructure-limit}
    u \mathds{1}_{Q_M} = \alpha u \mathds{1}_{Q_M} + 0\,,
  \end{equation}
  and hence $u=0$, since $\alpha \neq 0$. We have therefore found
  \begin{equation*}
    E^\eta = (0, 0, u^\eta) \to 0 \quad \text{ and } \quad H^\eta =
    (\nabla^{\perp} u^\eta, 0) \weakto 0 \quad \text{ in }L^2(Q_M;
    \mathbb{C}^3) \text{ as } \eta \to 0\,,
  \end{equation*}
  which was the claim.
\end{proof}

\subsubsection{Parallel magnetic field}
\label{sec:parall-magn-field}

We now consider a magnetic field that is parallel to $\e_3$,
$H^\eta(x) = (0, 0, u^\eta(x_1, x_2))$, with all quantities being
$x_3$-independent. This $H$-parallel case is the interesting case for
homogenization and it has the potential to generate magnetically
active materials. It was analyzed e.g.\,in \cite {BBF09hommaxwell,
  BBF15hommaxwell, BS13plasmonwaves, BF97homfibres}.  In this setting,
Maxwell's equations \eqref{eq:time-harmonic-Maxwell-eq} reduce to
\begin{subequations}
  \label{eq:Analysis-high-contrast-two-dimensional-Maxwell-system-parallel-H}
  \begin{empheq}[left=\empheqlbrace]{alignat=2}
    \nabla^{\perp} \cdot (E_1^\eta, E_2^\eta) & = \phantom{-} i \omega
    \mu_0 u^\eta &&\quad \text{ in } G \, , \label{eq:Analysis-high-contrast-two-dimensional-Maxwell-system-parallel-H-1}\\
    - \nabla^{\perp}u^\eta & = - i \omega \varepsilon_0
    \varepsilon_\eta (E_1^\eta, E_2^\eta) &&\quad \text{ in } G\, .
  \end{empheq}
\end{subequations}
System \eqref
{eq:Analysis-high-contrast-two-dimensional-Maxwell-system-parallel-H}
can equivalently be written as a scalar Helmholtz equation:
\begin{equation}
  \label{eq:Analysis-high-contrast-scalar-Helmholtz-parallel-H}
  - \nabla \cdot \bigg( \frac{1}{\varepsilon_\eta} \nabla u^\eta\bigg) =
  \omega^2 \varepsilon_0 \mu_0 u^\eta \quad \text{ in } G\, .
\end{equation}
In \eqref {eq:Analysis-high-contrast-E-parallel-scalar-Helmholtz}, the
high-contrast coefficient is outside the differential operator, which
induces a trivial limit behaviour of solutions.  Instead, \eqref
{eq:Analysis-high-contrast-scalar-Helmholtz-parallel-H} has the
high-contrast coefficient inside the differential operator, which
leads to a much richer behaviour of solutions.

The case $\Sigma = \Sigma_3$ is the metal plate (see Fig.~%
\ref{fig:the metal plate}) that was studied in
\cite{BS13plasmonwaves}.  The result of~\cite{BS13plasmonwaves} is the derivation of a
limit system with nontrivial solutions. In particular, the weak limit
of $(u^\eta)_{\eta}$ can be non-trivial. Similar results are available
for metallic wires $\Sigma = \Sigma_1$ (see Fig.~\ref{fig:the metal
  cylinder}); the results of \cite {BBF15hommaxwell} imply that also
in this case the weak limit of $(u^\eta)_{\eta}$ can be non-trivial.

We therefore observe that the $H$-parallel case does not allow to conclude
$H^\eta \weakto 0$ in $L^2(Q_M; \mathbb{C}^3)$. We note that
$H^\eta =(0, 0, u^\eta)\not\weakto 0$ implies, by boundedness of the
magnetic field and equation~\eqref{eq:Analysis-high-contrast-two-dimensional-Maxwell-system-parallel-H-1}, also
$E^\eta\not \weakto 0$.

\section{Finite element based multiscale approximation}
\label{sec:numerics}

In this section, we present numerical multiscale methods that are used
to study Maxwell's equations in high-contrast media from a numerical point of view.
We introduce the necessary notation for finite element discretisations and briefly
discuss the utilized approaches.  Based on these methods, numerical
experiments illustrating the transmission properties of the
microstructures are presented in Section \ref{subsec:experiment}.

\subsection{Variational problem for the second order formulation}
We study time-harmonic Maxwell's
equations in their second-order formulation for the magnetic field $H$
\eqref{eq:time-harmonic-Maxwell-eq-H}.  The macroscopic domain $\tilde G$ is
assumed to be bounded and we impose impedance boundary conditions on
the Lipschitz boundary $\partial \tilde G$ with the outer unit normal $n$:
\[\curl H\times n-ik_0(n \times H)\times n=g,\]
where $g\in L^2(\partial \tilde G)$ with $g\cdot n=0$ is given and
$k_0=\omega\sqrt{\varepsilon_0\mu_0}$ is the wavenumber.  These
boundary conditions can be interpreted as first-order approximation to
the Silver-M\"uller radiation conditions (used for the full space
$\rz^3$); the data $g$ are usually computed from an incident wave.
For the material parameters, we choose $\mu=1$ and $\varepsilon_\eta$
as specified in \ref{eq:Analysis-high-contrast-permittivity}.
Multiplying with a test function and integrating by parts results in
the following variational formulation: Find $H^\eta\in H_{\imp}(\tilde G)$
such that
\begin{align}
  \label{eq:numerics-maxwell-hfield-pde}
  \int_G\varepsilon_\eta^{-1}\curl H^\eta\cdot \curl \psi
  -k_0^2H^\eta\cdot \psi\, dx-ik_0\int_{\partial G}H^\eta_T\cdot
  \psi_T\, d\sigma=\int_{\partial G}g\cdot \psi_T\qquad \forall
  \psi\in H_{\imp}(\tilde G),
\end{align}
where $H_{\imp}(\tilde G):=\{v\in L^2(\tilde G; \cz^3)|\curl v\in L^2(\tilde G;
\cz^3),\,v_T\in L^2(\partial \tilde G)\}$ and $v_T:=v-(v\cdot n)n$ denotes
the tangential component.  Existence and uniqueness of the solution to
this problem for fixed $\eta$ is shown, for instance, in \cite{Monk}.

\subsection{Traditional finite element discretisation}  The standard finite
element discretisation of \eqref{eq:numerics-maxwell-hfield-pde} is a
Galerkin procedure with a finite-dimensional approximation space
$V_h\subset H_{\imp}(\tilde G)$ which consists of piecewise polynomial
functions on a (tetrahedral) mesh of $\tilde G$.  In detail, we denote by
$\CT_h=\{ T_j|j\in J\}$ a partition of $\tilde G$ into tetrahedra.  We assume
that $\CT_H$ is regular (i.e., no hanging nodes or edges occur), shape
regular (i.e., the minimal angle stays bounded under mesh refinement),
and that it resolves the partition into the meta-material $Q_M$ and
air $\tilde G\setminus \overline{Q}_M$.  To allow for such a partition, we
implicitly assume $\tilde G$ and $Q_M$ to be Lipschitz polyhedra.  Otherwise,
boundary approximations have to used which only makes the following
description more technical.  We define the local mesh size
$h_j:=\diam(T_j)$ and the global mesh size $h:=\max_{j\in J}h_j$.  As
conforming finite element space for $H_{\imp}(G)$, we use the lowest
order edge elements introduced by N\'ed\'elec, i.e.,
\[V_h:=\{v_h\in H_{\imp}(\tilde G)|v_h|_K(x)=a+b\times x \text{ with }a, b\in
\cz^3, \; \forall K\in \CT_h\}.\] It is well known (see \cite{Monk},
for instance), that the finite element method with this test and trial
space in \eqref{eq:numerics-maxwell-hfield-pde} yields a well-posed
discrete solution $H_h$.  Furthermore, the following a priori error
estimate holds
\[\|H^\eta-H_h\|_{H(\curl)}\leq Ch (\|H^\eta\|_{H^1(\tilde G)}+\|\curl
H^\eta\|_{H^1(\tilde G)}).\] For the setting of
\eqref{eq:numerics-maxwell-hfield-pde} as discussed in this paper,
however, two major problems arise.  First, due to the discontinuities
of the electric permittivity $\varepsilon^{-1}_\eta$ the necessary
regularity of $H^\eta$ is not available, see
\cite{BGL13regularitymaxwell, CDN99maxwellinterface,
  CD00maxwellsingularities}.  Second, even in the case of sufficient
regularity the right-hand side of the error estimate experiences a
blow-up with $\|H^\eta\|_{H^1(\tilde G)}+\|\curl H^\eta\|_{H^1(\tilde G)}\to \infty$
for $\eta\to 0$.  In other words, a typical solution of
\eqref{eq:numerics-maxwell-hfield-pde} is subject to (strong)
oscillations in $\eta$ such that its derivative does not remain
bounded for the periodicity length tending to zero.  As a consequence,
the error estimate has a $\eta$-dependent right-hand side of the type
$h\eta^{-1}$.  Therefore, convergence of standard finite element
discretisations is only to be expected in the asymptotic regime when
$h\ll\eta$, i.e., the mesh has to resolve the oscillations in the PDE
coefficients.  As discussed in the introduction, this can become
prohibitively expensive.

\subsection{Heterogeneous Multiscale Method}  As a remedy to these
limitations of the standard finite element method, we consider a
specific multiscale method.  The idea is to extract macroscopic
properties of the solution with $\eta$-independent complexity or
computational effort, respectively.  The basic idea directly
comes up from the effective equation
\eqref{eq:Analysis-high-contrast-effective-eq-H}: Since this effective
equation is independent of $\eta$, it can be discretised on a rather
coarse mesh $\CT_h$ without the need to resolve the $\eta$-scale,
i.e., we can have $h>\eta$.  This results in an approximation of the
homogenized solution $\hat{H}$, which contains important macroscopic
information of $H^\eta$.  However, for the discretisation of the
homogenized equation, the effective material parameters
$\hat{\varepsilon}$ and $\hat{\mu}$ need to be known, at least at the
(macroscopic) quadrature points.  This can be achieved by introducing
another, again $\eta$-independent, mesh $\CT_{h_Y}=\{S_l|l\in I\}$ of
the unit cube $Y$ with maximal mesh size $h_Y=\max_{l\in
  I}\diam(S_l)$.  We assume that $\CT_{h_Y}$ is regular and shape
regular and resolves the partition of $Y$ into $\Sigma$ and
$Y\setminus \overline{\Sigma}$.  Furthermore, $\CT_{h_Y}$ has to be
periodic in the sense that it can be wrapped to a regular
triangulation of the torus, i.e., no hanging nodes or edges over the
periodic boundary.  Note that $h_Y$ denotes the mesh size of the
triangulation of the unit cube. Thus, it is in no way related to
$\eta$ and can be of the same order as $h$.  Based on this mesh, the
cell problems occurring in the definition of $\hat{\varepsilon}$ and
$\hat{\mu}$ can be discretised with standard (Lagrange and
N\'ed\'elec) finite element spaces.  For details we refer to
\cite{Ver17hmmmaxwell}.  All in all, we can now compute the
homogenized solution $\hat{H}$ of
\eqref{eq:Analysis-high-contrast-effective-eq-H} as follows:
1. Compute discrete solutions of the cell problems (see
\cite{BBF15hommaxwell} or\cite{Ver17hmmmaxwell}) using the mesh
$\CT_{h_Y}$ and the associated standard finite element
spaces. 2. Compute the effective parameters $\hat{\varepsilon}$ (or
$\widehat{\varepsilon^{-1}}$) and $\hat{\mu}$ approximatively with the
discrete cell problems solutions. 3. Compute the discrete homogenized
solution of \eqref{eq:Analysis-high-contrast-effective-eq-H} with the
approximated effective coefficients and using the mesh $\CT_h$ with
the associated finite element space $V_h$ as introduced above.

This (naive) discretisation scheme for the effective equation
\eqref{eq:Analysis-high-contrast-effective-eq-H} in fact can be
interpreted as a specification of the Heterogeneous Mutliscale Method
(HMM) in the perfectly periodic case.  The Finite Element
Heterogeneous Multiscale Method, introduced by E and Enguist
\cite{EE03hmm, EE05hmm}, sets up a macroscopic sesquilinear form to
compute the HMM solution $H_h$, which is an approximation of the
homogenized solution $\hat{H}$.  The macroscopic sesquilinear form is
very similar to the effective sesquilinear form associated with the
left-hand side of \eqref{eq:Analysis-high-contrast-effective-eq-H},
but the effective material parameters are not computed a priori.
Instead local variants of the cell problems are set up on
$\eta$-scaled cubes $Y_j^\eta=\eta Y+x_j$ around macroscopic
quadrature points $x_j$.  We can still use the mesh $\CT_{h_Y}$ of the
unit cube $Y$ and transform it to a partition $\CT_{h_Y}(Y_j^\eta)$ of
the scaled unit cell $Y_j^\eta$.  Similarly, also the finite element
spaces associated with $\CT_{h_Y}$ can be transferred to spaces on
$\CT_{h_Y}(Y_j^\eta)$ using a suitable affine mapping.  The finescale
computations result in so called local reconstructions, which consist
of the macroscopic basis functions and the corresponding (discrete)
cell problem solutions.  Averages (over $Y_j^\eta$) of these local
reconstructions then enter the macroscopic sesquilinear form.  A
detailed definition of the HMM for Maxwell's equations in
high-contrast media is presented in \cite{Ver17hmmmaxwell}, where also
the connection to analytical homogenization as well as the possibility
to treat more general than purely periodic problems are discussed.  We
only want to emphasize one important feature of the HMM in
\cite{Ver17hmmmaxwell}: Apart from the (macroscopic) approximation
$H_h$, discrete correctors $H_{h_Y, 1}$, $H_{h_Y, 2}$, and $H_{h_Y,
  3}$ can be determined from the discrete cell problems (in a second
post-processing step).  Via these correctors, we can define the zeroth
order $L^2$-approximation $H^0_{\HMM}:=H_h+\nabla_y H_{h_Y, 2}(\cdot,
\frac{\cdot}{\delta})+H_{h_Y,3}(\cdot, \frac{\cdot}{\delta})$, which
corresponds to the first term of an asymptotic expansion and is used
to approximate the true solution $H^\eta$.  We again refer to
\cite{Ver17hmmmaxwell} for details and note that it has been observed
in several numerical examples that these correctors are a vital part
of the HMM-approximation, see \cite{HOV15maxwellHMM, GHV17lodcurl,
  OV16hmmhelmholtz, Ver17hmmmaxwell}.  We close by remarking that in
Section \ref{subsec:experiment} below, we extend the described HMM of
\cite{Ver17hmmmaxwell} to general microstructures although the
validity of the homogenized models in these cases is not shown so far,
see the discussion in Sections \ref{sec:homogenization} and
\ref{sec:high-contrast-media}.

\section{Numerical study of transmission properties for high-contrast
  inclusions}
\label{subsec:experiment}

In this section, we numerically study the transmission properties in
the case of high-contrast for the three micro-geometries: the metal
cylinder, the metal plate, and the air cylinder.  Since the aim of
this paper is a better understanding of the different microstructures
and their effect, we focus on the qualitative behaviour rather than
explicit convergence rates.  The implementation was done with the
module {\sffamily dune-gdt} \cite{wwwdunegdt} of the DUNE software
framework \cite{BB+08dune1, BB+08dune2}.

\smallskip \textbf{Setting.}  We consider Maxwell's equations in the
second-order formulation for the $H$-field
\eqref{eq:numerics-maxwell-hfield-pde} with a high-contrast medium as
defined in \eqref{eq:Analysis-high-contrast-permittivity}.  It remains
to specify the macroscopic geometry, the boundary date $g$, the
material parameter $\eps_1$, and the frequency.  We use a slab-like
macroscopic geometry similar to Section \ref{sec:geometry}, but we
truncate $G$ also in the $x_1$-direction to have a finite
computational domain $\tilde G$, as described in the previous section.
We choose $\tilde G=(0,1)^3$ with the meta-material located in
$Q_M=\{x\in G|0.25\leq x_1\leq 0.75\}$.  Note that $Q_M$ is translated
in $x_1$-direction compared to Section \ref{sec:geometry}, but this
does not influence the qualitative results of the analysis.  As in
Section \ref{sec:geometry}, we assume that an incident wave
$H_{\mathrm{inc}}$ from the right travels along the $x_1$-axis to the
left, i.e., $H_{\mathrm{inc}}=\exp(-ikx_1)p$ with a normalized
polarization vector $p\perp \Ve_1$.  This incident wave is used to
compute the boundary data $g$ as $g=\curl H_{\mathrm{inc}}\times
n-ik_0 n\times (H_{\mathrm{inc}}\times n)$.  We choose the inverse
permittivity as $\varepsilon_1^{-1}=1.0-0.01i$ and note that
$\varepsilon_1$ is only slightly dissipative.  In all experiments, we
choose the same wavenumber $k_0=12$ and the periodicity parameter
$\eta=1/8$.

As explained in the previous section, we want to use the Heterogeneous
Multiscale Method to obtain good approximations with reasonable
computational effort.  We use the mesh sizes $h_y=h=\sqrt{3}\cdot
1/16$ and compute the macroscopic HMM-approximation $H_h$ as well as
the zeroth order approximation $H_{\HMM}^0$, which also utilizes
information of the discrete correctors.  To demonstrate the validity
of the HMM, we use two different reference solutions.  First, the
homogenized reference solution $\Heff$ is computed as solution to
\eqref{eq:Analysis-high-contrast-effective-eq-H} on a mesh with size
$h=\sqrt{3}\cdot 1/48$, where the effective material parameters are
calculated approximatively using a discretisation of the unit cube
with mesh size $h_Y=\sqrt{3}\cdot 1/24$.  Second, the (true) reference
solution $H^\eta$ is computed as direct finite element discretisation
of \eqref{eq:numerics-maxwell-hfield-pde} on a fine grid with
$h_{\mathrm{ref}}=\sqrt{3}\cdot 1/64$.

\begin{table}
  \caption{Summary of analytical predictions of the transmission properties
    and references to numerical results. The first row provides the geometry.
    The second row indicates possible transmission polarizations (of $H$) according
    to the theory of perfect conductors of Section
    \ref {sec:Calculation-of-transmission-and-reflection-coefficient}.
    The third row indicates the possibility of transmission based on
    Section \ref {sec:high-contrast-media}: We mention cases in which we
    cannot derive weak convergence to $0$. An entry \enquote{-} indicates that no analytical
    result can be applied.  The last row provides the reference
    to the  visualization of the numerical calculation for high-contrast media.}
\label{tab:numsummary}
\centering
\begin{tabular}{@{}ccccc@{}}
  \toprule
  geometry
  &metal cylinder $\Sigma_1$&metal cylinder $\Sigma_2$&metal plate $\Sigma_3$
  &air cyl. $\Sigma_4$\\
  \midrule
  transmission (PC)& $\Ve_3$-polarized&$\Ve_2$ and $\Ve_3$-polarized &
                                                                       $\Ve_3$-polarized&no\\
  \midrule
  nontriv. limit (HC)& $\Ve_3$-polarized&- &$\Ve_3$-polarized&-\\
  \midrule
  numerical example& Fig.~\ref{fig:Sigma1-homrefsol}& Fig.~\ref{fig:Sigma2-homrefsol}
  & Fig.~\ref{fig:Sigma3-homrefsol} & Fig.~\ref{fig:Sigma4-refsol}\\
  \bottomrule
\end{tabular}
\end{table}

\smallskip \textbf{Main results.}  Before we discuss the examples in
detail, we present an overview of the results.  The qualitative
transmission properties of the meta-material are in good agreement
with the theory of Section
\ref{sec:Calculation-of-transmission-and-reflection-coefficient},
although the numerical examples consider high-contrast media instead
of perfect conductors.  The predictions and the corresponding
numerical examples are summarized in Table \ref{tab:numsummary}.  In
contrast to perfect conductors, the high-contrast medium leads to
rather high intensities and amplitudes of the $H$-field inside the
inclusions $\Sigma_\eta$.  Depending on the chosen wavenumber,
Mie-resonances inside the inclusions can occur for high-contrast
media; see Section \ref{sec:high-contrast-media} and
\cite{BBF15hommaxwell, Ver17hmmmaxwell}.  Our numerical experiments
also show that the HMM yields (qualitatively) good approximations,
although the validity of the underlying effective models is not proved
for the studied geometries.

\subsection{Metal cuboids $\tilde \Sigma_1$ and $\tilde \Sigma_2$}

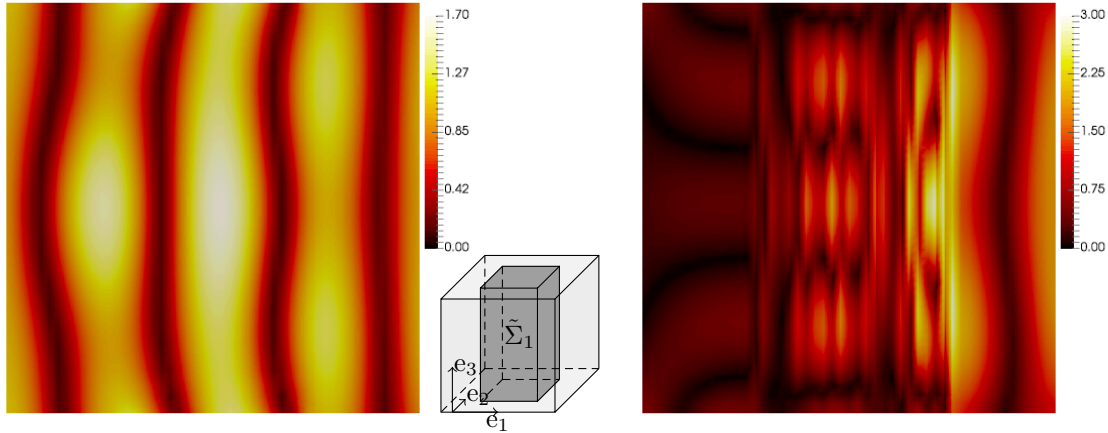
\begin{figure}[t]
\includegraphics[width=0.44\textwidth, trim= 75mm 33mm 51mm 32mm, clip=true, keepaspectratio=false]{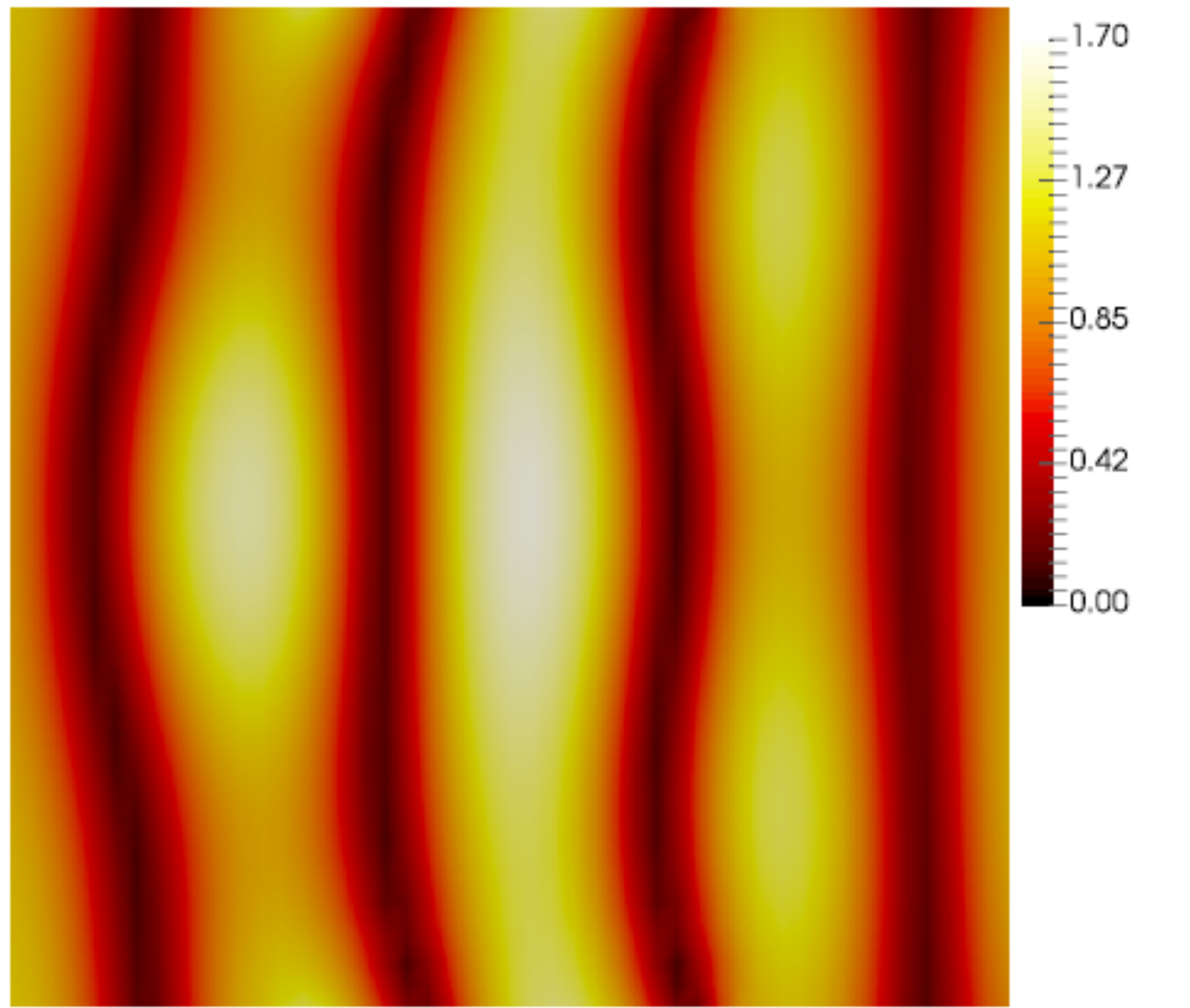}%
\hspace{1.9cm}%
\includegraphics[width=0.44\textwidth, trim= 75mm 33mm 51mm 32mm, clip=true, keepaspectratio=false]{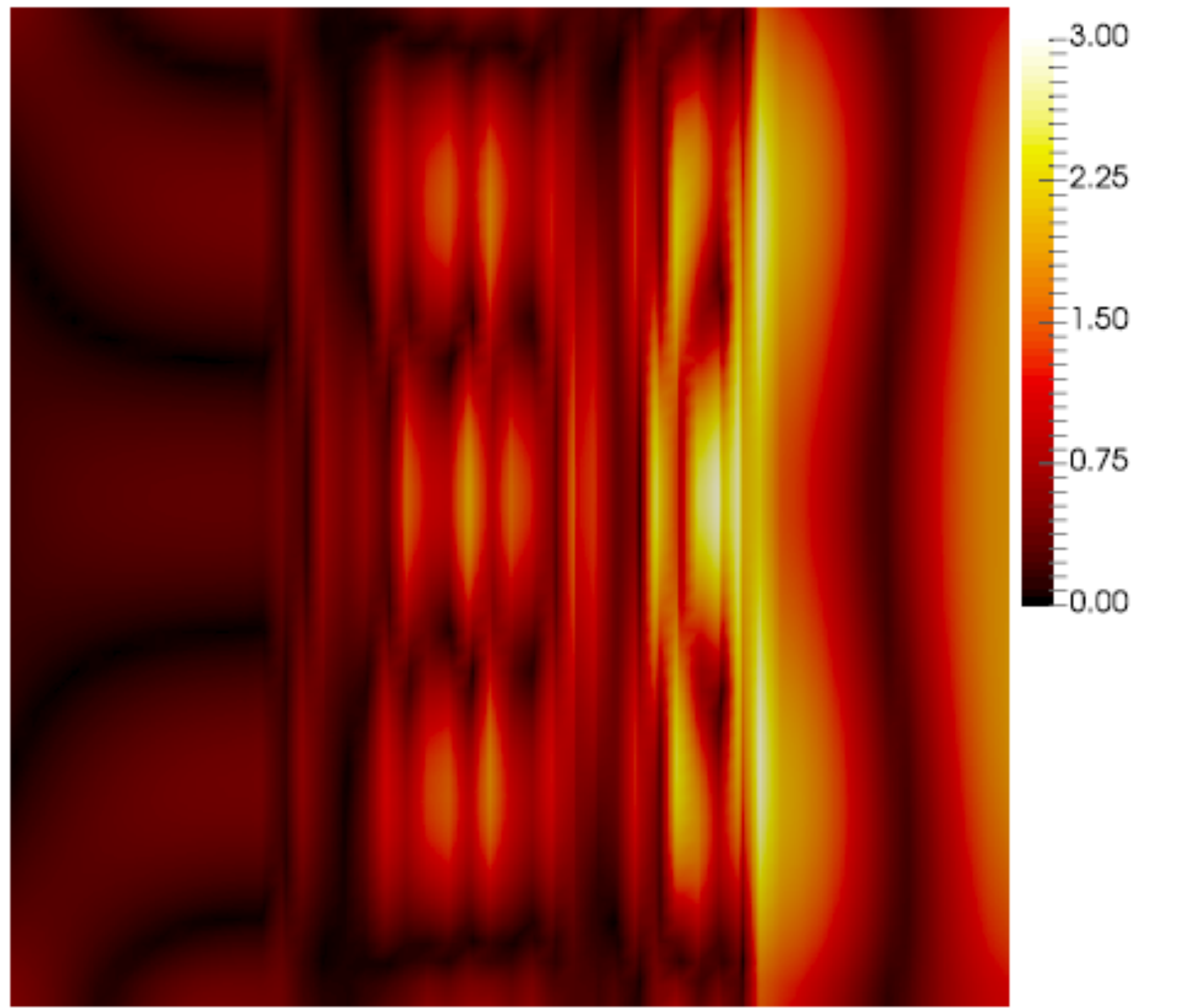}
  \begin{tikzpicture}[scale=1.5, overlay, xshift=30ex, yshift=6.5ex]
  \coordinate (A) at (-0.5, -0.5, 0.5);
  \coordinate (B) at (0.5,-0.5,0.5);
  \coordinate (C) at (0.5,0.5,0.5);
  \coordinate (D) at (-0.5, 0.5, 0.5);

  \coordinate (E) at (-0.5, 0.5, -0.5);
  \coordinate (F) at (0.5, 0.5, -0.5);
  \coordinate (G) at (0.5, -0.5, -0.5);
  \coordinate (H) at (-0.5, -0.5, -0.5);

  \coordinate (A1) at (-0.25, -0.5, 0.25);
  \coordinate (B1) at (0.25,-0.5,0.25);
  \coordinate (C1) at (0.25,0.5,0.25);
  \coordinate (D1) at (-0.25, 0.5, 0.25);

  \coordinate (E1) at (-0.25, 0.5, -0.25);
  \coordinate (F1) at (0.25, 0.5, -0.25);
  \coordinate (G1) at (0.25, -0.5, -0.25);
  \coordinate (H1) at (-0.25, -0.5, -0.25);

  \fill[gray!20!white, opacity=.5] (A) -- (B) -- (C) -- (D) -- cycle;
  \fill[gray!20!white, opacity=.5] (E) -- (F) -- (G) --(H) -- cycle;
  \fill[gray!20!white, opacity=.5] (D) -- (E) -- (H) -- (A)-- cycle;
  \fill[gray!20!white, opacity=.5] (B) -- (C) -- (F) -- (G) -- cycle;

  \fill[gray!75!white] (A1)--(B1)--(C1)--(D1)--cycle;
  \filldraw[gray!75!white] (E1)--(F1)--(G1)--(H1)--cycle;
  \filldraw[gray!75!white] (D1)--(E1)--(H1)--(A1)--cycle;
  \filldraw[gray!75!white] (B1)--(C1)--(F1)--(G1)--cycle;

  \draw[] (A) -- (B) -- (C) -- (D) --cycle (E) -- (D) (F) -- (C) (G) -- (B);
  \draw[] (E) -- (F) -- (G) ;
  \draw[densely dashed] (E) -- (H) (H) -- (G) (H) -- (A);

  \draw[] (A1)--(B1)--(C1)--(D1)--cycle (E1)--(D1) (F1)--(C1) (G1)--(B1);
  \draw[] (E1)--(F1)--(G1);
  \draw[densely dashed] (E1)--(H1) (H1)--(G1) (H1)--(A1);

  \draw[->] ( -.4, -.5, .5)--( 0, -.5, .5 );
  \draw[->] (-.4, -.5, .5) -- (-.4, -.1, .5);
  \draw[->] (-.4, -.5, .5) -- (-.4, -.5, .2);

  \node[] at (0,0,0) {$\tilde \Sigma_1$};
\node[] at (-0, -.6, .5){$\e_1$};
  \node[] at (-.28, -.1, .5){$\e_3$};
  \node[] at (-.4, -.6, .-.1){$\e_2$};

\end{tikzpicture}
\caption{Metal cuboid $\tilde \Sigma_1$, the magnitude of $\Re(\Heff)$ is
  plotted. Left: The $H$-field is $\e_3$-polarized and the plot shows
  values in the plane $x_3=0.5$. The analysis of both, (PC) and (HC)
  yields: transmission is possible. Right: The $H$-field is
  $\e_2$-polarized and the plot shows values in the plane $x_2 =
  0.545$. Since the $H$-field is not parallel to $\e_3$, the analysis
  of (PC) and (HC) predicts that no transmission is possible. Inlet in
  the middle: Microstructure in the unit cube.}
\label{fig:Sigma1-homrefsol}
\end{figure}

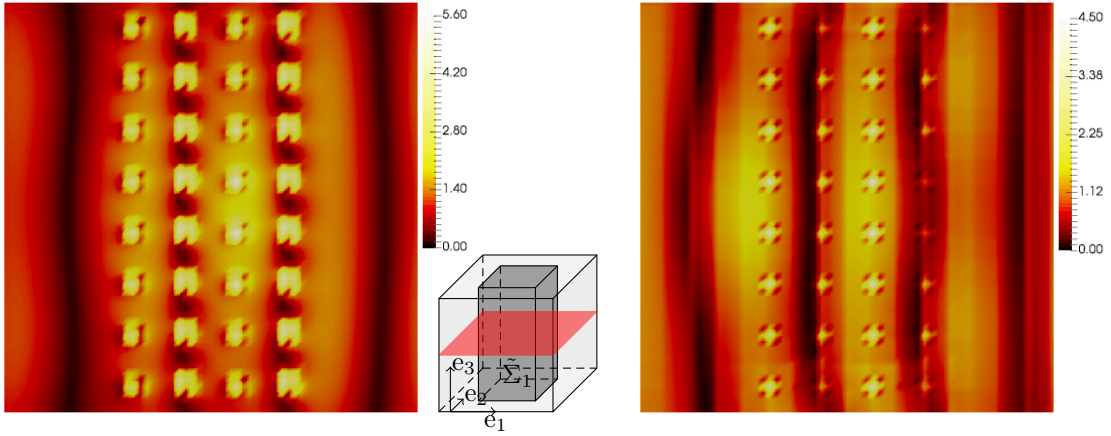
\begin{figure}
\includegraphics[width=0.44\textwidth, trim= 75mm 33mm 51mm 32mm, clip=true, keepaspectratio=false]{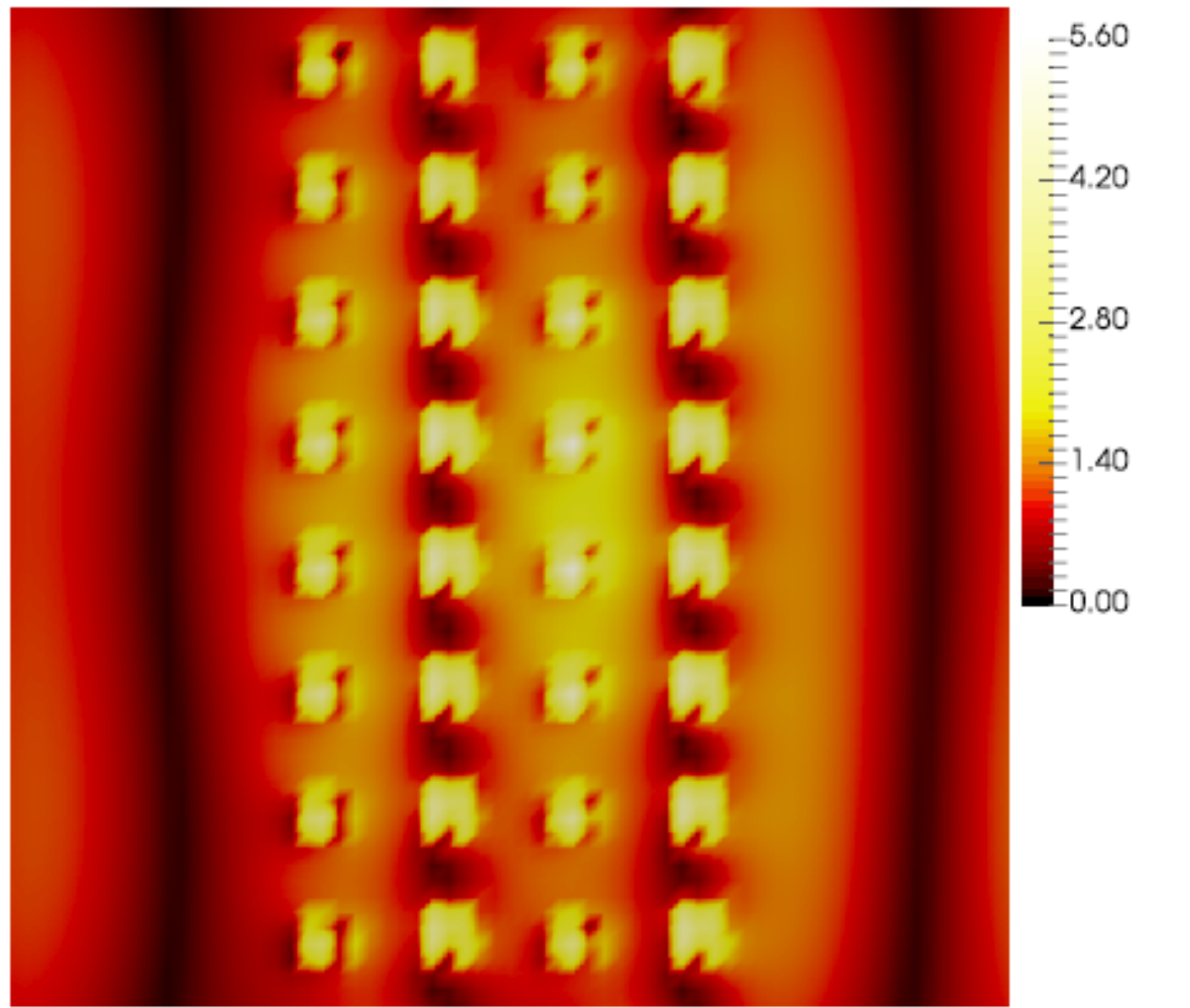}%
\hspace{1.9cm}%
\includegraphics[width=0.44\textwidth, trim= 75mm 33mm 51mm 32mm, clip=true, keepaspectratio=false]{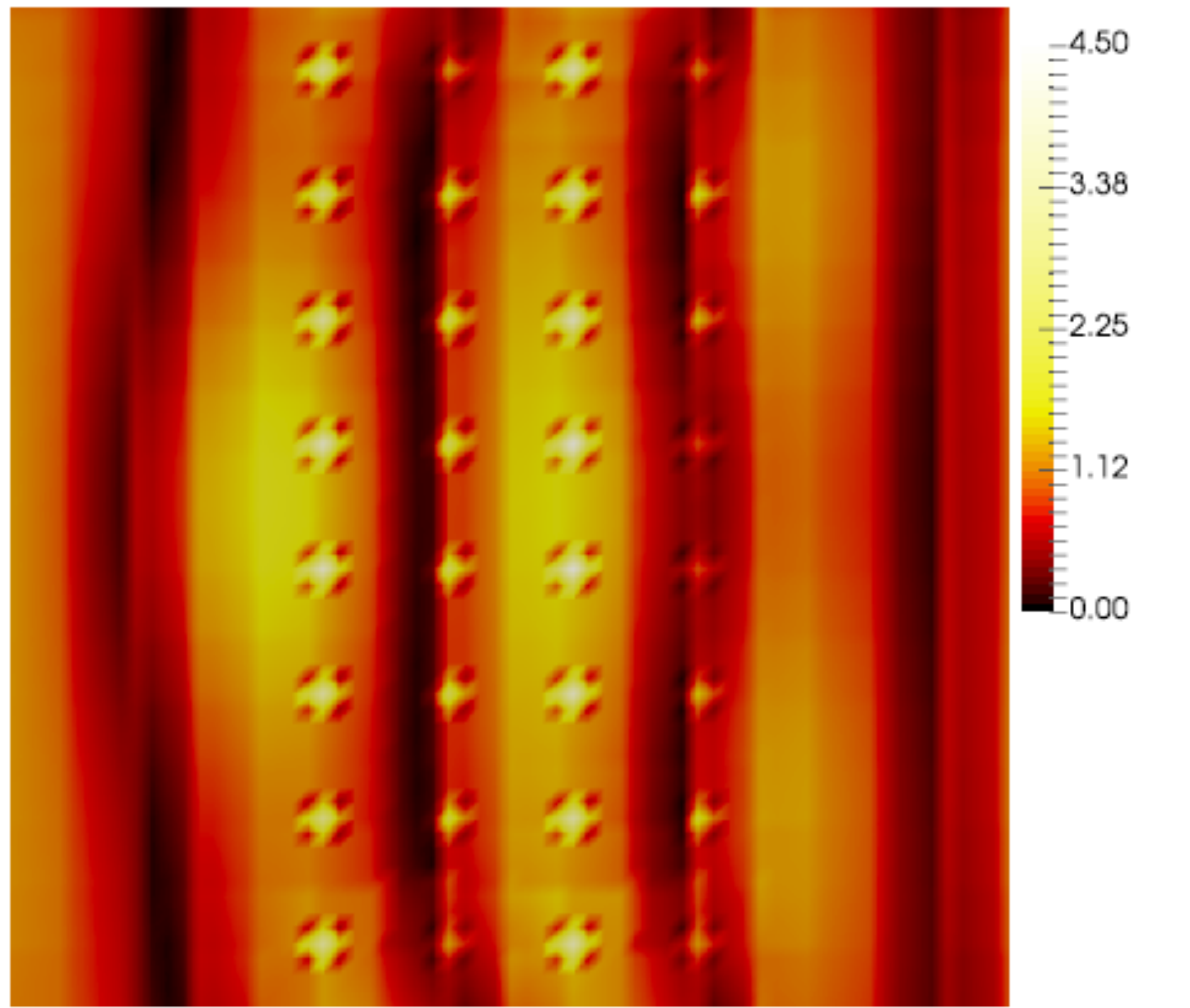}
  \begin{tikzpicture}[scale=1.5, overlay, xshift=30ex, yshift=6.5ex]
  \coordinate (A) at (-0.5, -0.5, 0.5);
  \coordinate (B) at (0.5,-0.5,0.5);
  \coordinate (C) at (0.5,0.5,0.5);
  \coordinate (D) at (-0.5, 0.5, 0.5);

  \coordinate (E) at (-0.5, 0.5, -0.5);
  \coordinate (F) at (0.5, 0.5, -0.5);
  \coordinate (G) at (0.5, -0.5, -0.5);
  \coordinate (H) at (-0.5, -0.5, -0.5);

  \coordinate (A1) at (-0.25, -0.5, 0.25);
  \coordinate (B1) at (0.25,-0.5,0.25);
  \coordinate (C1) at (0.25,0.5,0.25);
  \coordinate (D1) at (-0.25, 0.5, 0.25);

  \coordinate (E1) at (-0.25, 0.5, -0.25);
  \coordinate (F1) at (0.25, 0.5, -0.25);
  \coordinate (G1) at (0.25, -0.5, -0.25);
  \coordinate (H1) at (-0.25, -0.5, -0.25);

  \fill[gray!20!white, opacity=.5] (A) -- (B) -- (C) -- (D) -- cycle;
  \fill[gray!20!white, opacity=.5] (E) -- (F) -- (G) --(H) -- cycle;
  \fill[gray!20!white, opacity=.5] (D) -- (E) -- (H) -- (A)-- cycle;
  \fill[gray!20!white, opacity=.5] (B) -- (C) -- (F) -- (G) -- cycle;

  \fill[gray!75!white] (A1)--(B1)--(C1)--(D1)--cycle;
  \filldraw[gray!75!white] (E1)--(F1)--(G1)--(H1)--cycle;
  \filldraw[gray!75!white] (D1)--(E1)--(H1)--(A1)--cycle;
  \filldraw[gray!75!white] (B1)--(C1)--(F1)--(G1)--cycle;

  \draw[] (A) -- (B) -- (C) -- (D) --cycle (E) -- (D) (F) -- (C) (G) -- (B);
  \draw[] (E) -- (F) -- (G) ;
  \draw[densely dashed] (E) -- (H) (H) -- (G) (H) -- (A);

  \draw[] (A1)--(B1)--(C1)--(D1)--cycle (E1)--(D1) (F1)--(C1) (G1)--(B1);
  \draw[] (E1)--(F1)--(G1);
  \draw[densely dashed] (E1)--(H1) (H1)--(G1) (H1)--(A1);

  \draw[->] ( -.4, -.5, .5)--( 0, -.5, .5 );
  \draw[->] (-.4, -.5, .5) -- (-.4, -.1, .5);
  \draw[->] (-.4, -.5, .5) -- (-.4, -.5, .2);

  \node[] at (0,-0.35,0) {$\tilde \Sigma_1$};
\node[] at (-0, -.6, .5){$\e_1$};
  \node[] at (-.28, -.1, .5){$\e_3$};
  \node[] at (-.4, -.6, .-.1){$\e_2$};

  \filldraw[red, opacity=.5] (-0.5, 0, -0.5) -- (0.5, 0, -0.5) --(0.5, 0, 0.5) --(-0.5, 0, 0.5)--cycle;

\end{tikzpicture}
\caption{Test of numerical schemes for the metal cuboid
  $\tilde \Sigma_1$. We consider an $\e_3$-polarized incoming $H$-field and
  plot the solution in the plane $x_3=0.5$; the colors indicate the
  magnitude of the reference solution $\Re(H^\eta)$ (left) and the
  zeroth order approximation $\Re(H^0_{\HMM})$ (right). Inlet in the
  center: Microsctructure in the unit cube with visualization plane in
  red.}
\label{fig:Sigma1-HMM}
\end{figure}

\begin{figure}
  \includegraphics[width=0.44\textwidth, trim= 75mm 33mm 51mm 32mm,
  clip=true,
  keepaspectratio=false]{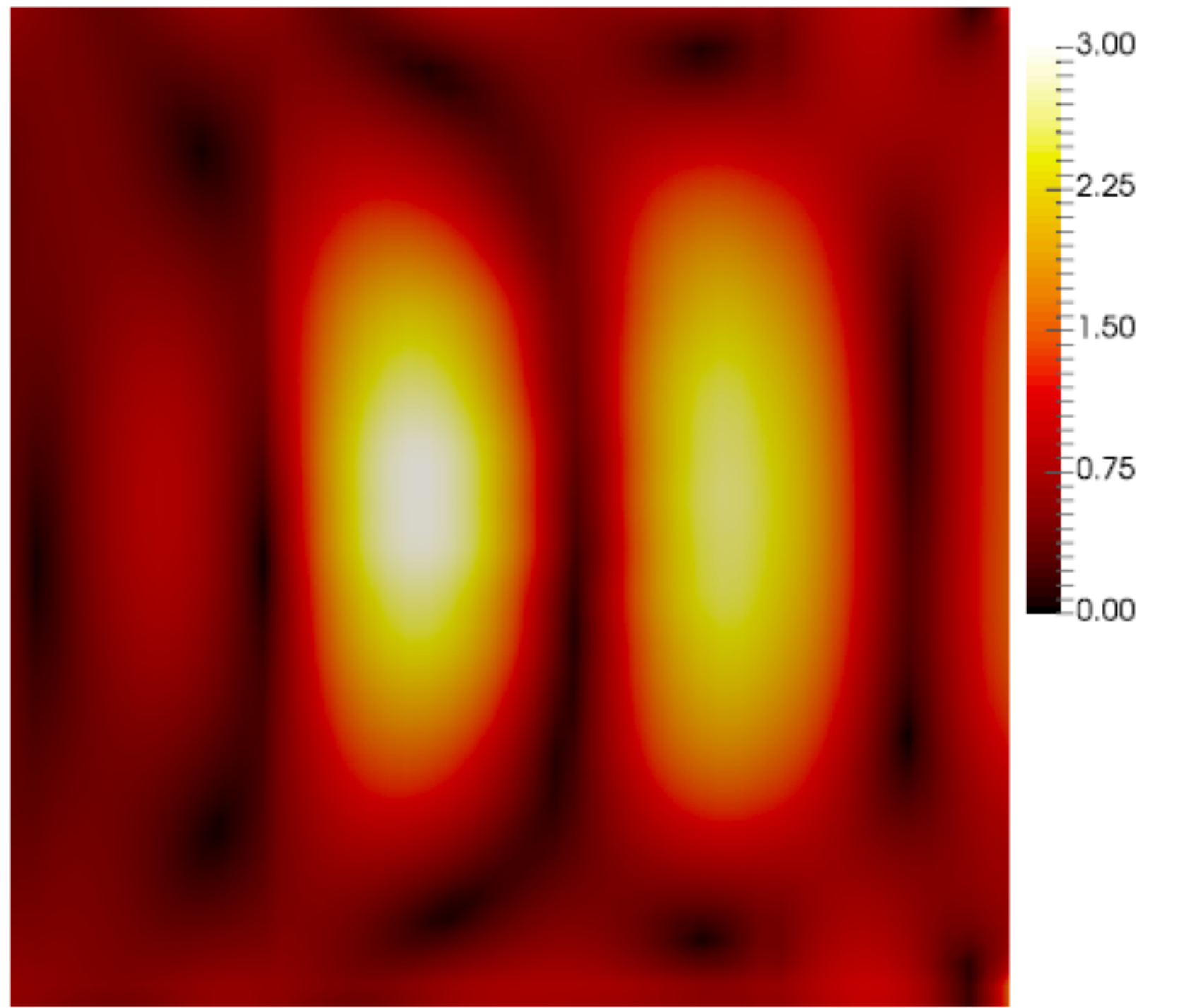}%
  \hspace{1.9cm}%
  \includegraphics[width=0.44\textwidth, trim= 75mm 33mm 51mm 32mm,
  clip=true,
  keepaspectratio=false]{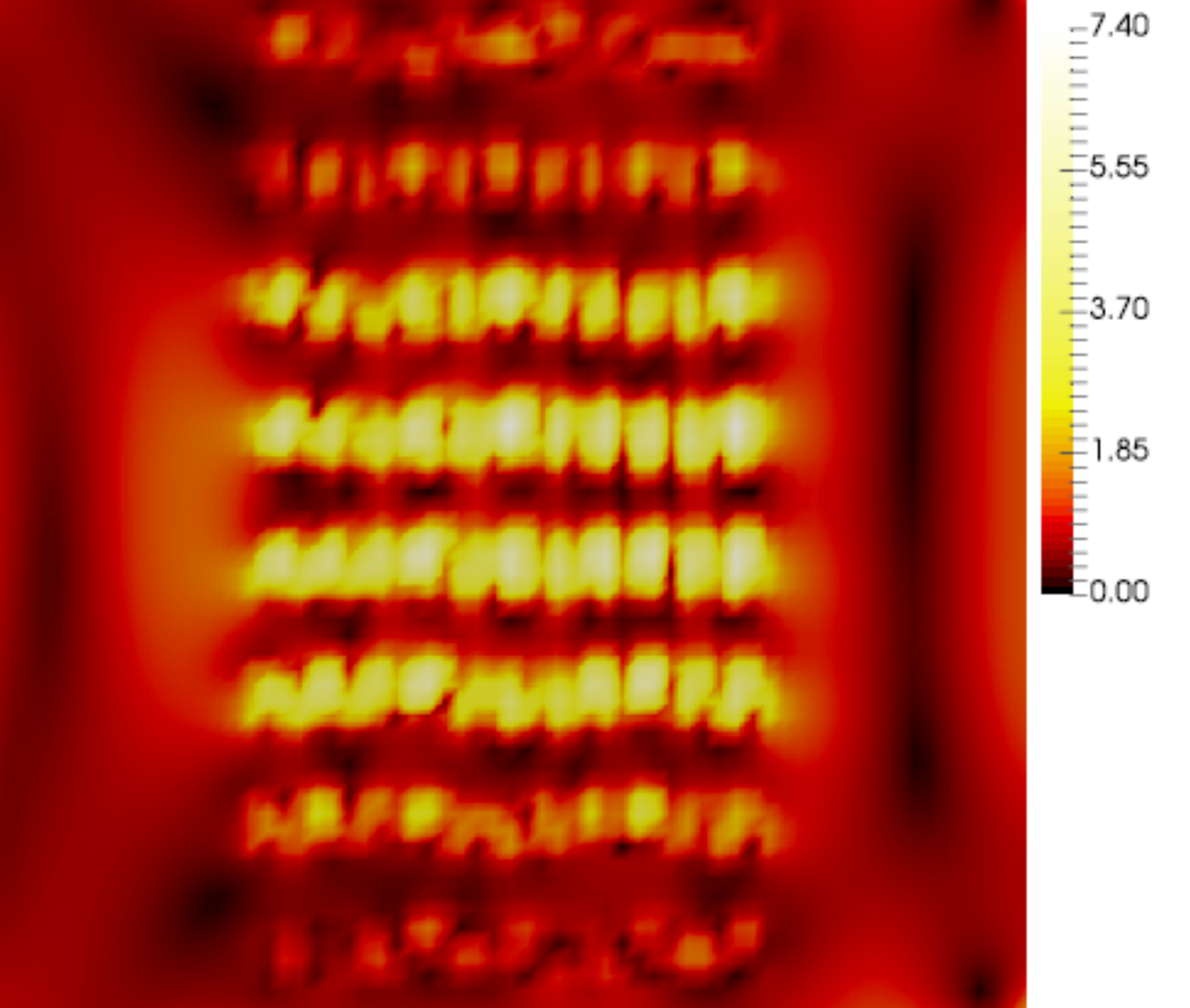}
  \begin{tikzpicture}[scale=1.5, overlay, xshift=30ex, yshift=6.5ex]
  \coordinate (A) at (-0.5, -0.5, 0.5);
  \coordinate (B) at (0.5,-0.5,0.5);
  \coordinate (C) at (0.5,0.5,0.5);
  \coordinate (D) at (-0.5, 0.5, 0.5);

  \coordinate (E) at (-0.5, 0.5, -0.5);
  \coordinate (F) at (0.5, 0.5, -0.5);
  \coordinate (G) at (0.5, -0.5, -0.5);
  \coordinate (H) at (-0.5, -0.5, -0.5);

  \coordinate (A1) at (-0.5, -0.25, 0.25);
  \coordinate (B1) at (0.5,-0.25,0.25);
  \coordinate (C1) at (0.5,0.25,0.25);
  \coordinate (D1) at (-0.5, 0.25, 0.25);

  \coordinate (E1) at (-0.5, 0.25, -0.25);
  \coordinate (F1) at (0.5, 0.25, -0.25);
  \coordinate (G1) at (0.5, -0.25, -0.25);
  \coordinate (H1) at (-0.5, -0.25, -0.25);

  \fill[gray!20!white, opacity=.5] (A) -- (B) -- (C) -- (D) -- cycle;
  \fill[gray!20!white, opacity=.5] (E) -- (F) -- (G) --(H) -- cycle;
  \fill[gray!20!white, opacity=.5] (D) -- (E) -- (H) -- (A)-- cycle;
  \fill[gray!20!white, opacity=.5] (B) -- (C) -- (F) -- (G) -- cycle;

  \fill[gray!75!white] (A1)--(B1)--(C1)--(D1)--cycle;
  \filldraw[gray!75!white] (E1)--(F1)--(G1)--(H1)--cycle;
  \filldraw[gray!75!white] (D1)--(E1)--(H1)--(A1)--cycle;
  \filldraw[gray!75!white] (B1)--(C1)--(F1)--(G1)--cycle;

  \draw[] (A) -- (B) -- (C) -- (D) --cycle (E) -- (D) (F) -- (C) (G) -- (B);
  \draw[] (E) -- (F) -- (G) ;
  \draw[densely dashed] (E) -- (H) (H) -- (G) (H) -- (A);

  \draw[] (A1)--(B1)--(C1)--(D1)--cycle (E1)--(D1) (F1)--(C1) (G1)--(B1);
  \draw[] (E1)--(F1)--(G1);
  \draw[densely dashed] (E1)--(H1) (H1)--(G1) (H1)--(A1);

  \draw[->] ( -.4, -.5, .5)--( 0, -.5, .5 );
  \draw[->] (-.4, -.5, .5) -- (-.4, -.2, .5);
  \draw[->] (-.4, -.5, .5) -- (-.4, -.5, .2);

  \node[] at (0,-0,0) {$\tilde \Sigma_2$};
\node[] at (-0, -.6, .5){$\e_1$};
  \node[] at (-.2, -.2, .5){$\e_3$};
  \node[] at (-.4, -.6, .-.1){$\e_2$};

  \filldraw[red, opacity=.5] (-0.5, -0.5, -0) -- (0.5, -0.5, -0) --(0.5, 0.5, 0) --(-0.5, 0.5, 0)--cycle;

\end{tikzpicture}

\caption{Metal cuboid $\tilde \Sigma_2$. We study an $\e_3$-polarized
  incident $H$-field and plot the magnitude of $\Re(\Heff)$ (left) and
  $\Re(H^\eta)$ (right) in the plane $x_2=0.545$. The analysis (PC)
  predicts transmission in this case, the analysis (HC) does not
  exclude transmission.  Middle: Microstructure in the unit cube with
  visualization plane in red.}
\label{fig:Sigma2-homrefsol}
\end{figure}

Instead of metal cylinders with circular base we study metal cuboids
with square base, so that we do not have to deal with boundary
approximations in our numerical method. We choose $\tilde
\Sigma_1=(0.25, 0.75)^2\times (0,1)$ and $\tilde \Sigma_2=(0,1)\times
(0.25, 0.75)^2$.  Note that this choice influences the value of
$\gamma$, but not the other results of Sections
\ref{sec:metal-cylinder} and \ref{sec:metal-cylinder-case}.  Due to
the symmetry of the microstructure, the effective material parameters
are diagonal matrices with $a_{11}=a_{22}$, but with a different value
$a_{33}$; see the analytical computations in Section
\ref{sec:metal-cylinder}.  Up to numerical errors, we obtain the same
structure for the computed approximative effective parameters.

Comparing the homogenized reference solution $\Heff$ for $\e_2$- and
$\e_3$-polarized incoming waves for $\tilde \Sigma_1$ in
Fig.~\ref{fig:Sigma1-homrefsol}, we observe that the $\e_3$-polarized
wave is transmitted almost undisturbed through the meta-material.  For
the $\e_2$-polarization, however, the field intensity in $Q_L\coloneqq
\set{x \in G \given x\leq 0.25}$ is very low, corresponding to small
transmission factors.  This matches the analytical predictions of
Section \ref{sec:metal-cylinder}, which yields transmission only for
$\e_3$-parallel $H$-fields. The same effect is predicted for
high-contrast media by the analysis of Section
\ref{sec:parall-electr-field}.

The HMM can reproduce the behaviour of the homogenized and of the
heterogeneous solution.  For the comparison, we only consider the
$\e_3$-polarized incoming wave in Fig.~\ref{fig:Sigma1-HMM} and
compare the zeroth order approximation $H_{\HMM}^0$ (right) to the
(true) reference solution $H^\eta$ (left).  Errors are still visible,
but the qualitative agreement is good, even for the coarse mesh size
of $h=h_Y=\sqrt{3}\cdot 1/16$ chosen for the HMM.  In particular, the
rather cheaply computable zeroth order approximation $H^0_{\HMM}$ can
capture most of the important features of the \emph{true} solution,
even for inclusions of high-contrast. This clearly underlines the
potential of the HMM.  Moreover, Fig.~\ref{fig:Sigma1-HMM} underlines
the specific behaviour of $H^\eta$ in the inclusions $\tilde
\Sigma_\eta$ for high-contrast media.  As analyzed in Section
\ref{sec:parall-magn-field}, $H^\eta$ cannot be expected to vanish in
the inclusions in the limit $\eta\to 0$ due to possible resonances;
see \cite{BBF09hommaxwell}.  We observe rather high field intensities
in the inclusions; see also \cite{Ver17hmmmaxwell} for a slightly
different inclusion geometry.

\smallskip We also study the rotated metal cuboid $\tilde \Sigma_2$.  In
correspondence to the analysis of Section
\ref{sec:metal-cylinder-case}, we observe transmission for an
$\e_3$-polarized incident wave; see Fig.~\ref{fig:Sigma2-homrefsol}.
Note that the homogenized reference solution looks different to
$\tilde \Sigma_1$ because of the rotation of the geometry, which is also
reflected in the different structure and values of the reflection and
transmission coefficients.  The (true) reference solution in Fig.~%
\ref{fig:Sigma3-homrefsol} shows the high field intensities in the
metal cuboids induced by the high-contrast permittivity.

\subsection{Metal plate $\Sigma_3$}

\begin{figure}
  \includegraphics[width=0.44\textwidth, trim= 75mm 33mm 51mm 32mm,
  clip=true,
  keepaspectratio=false]{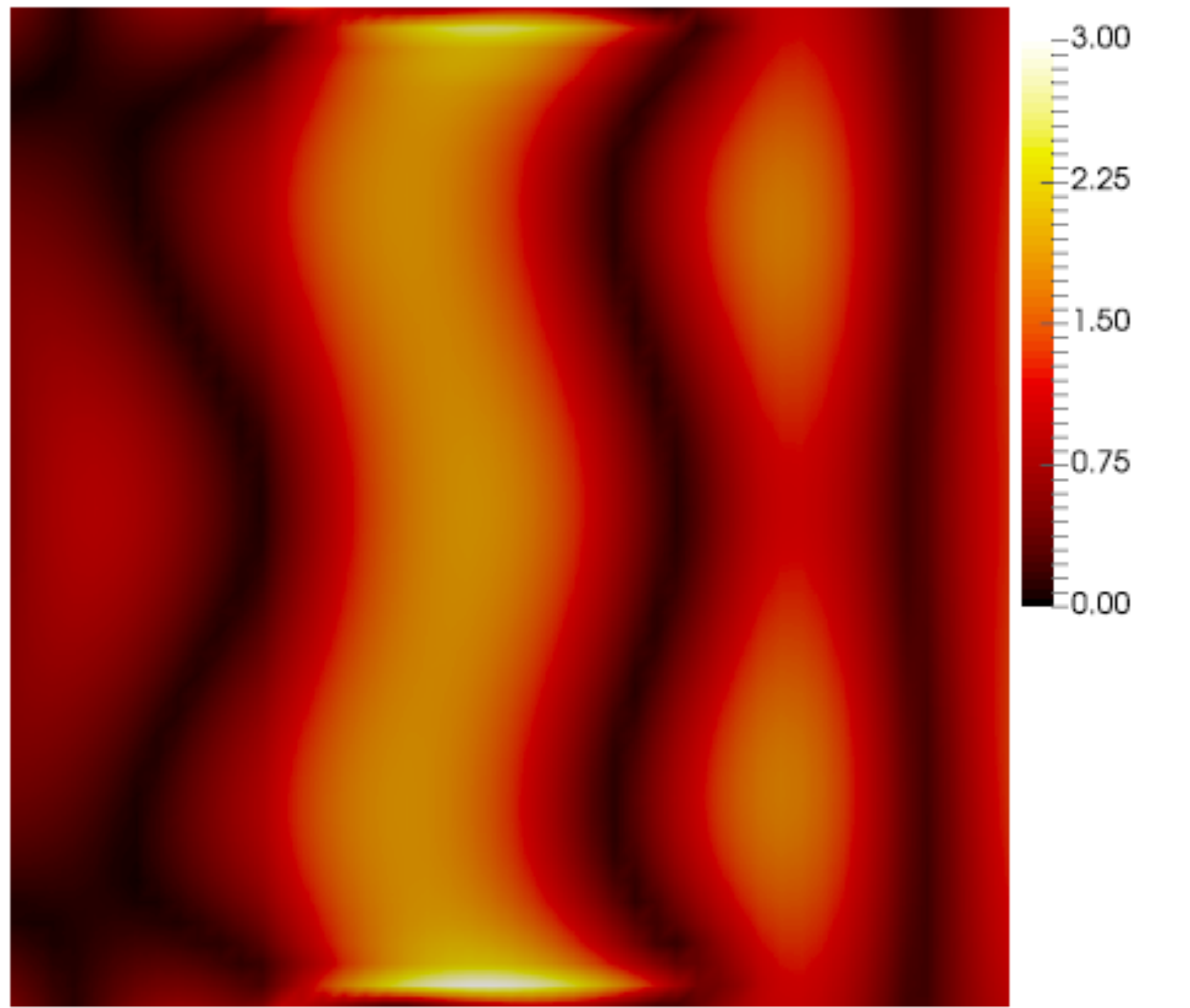}%
  \hspace{1.9cm}%
  \includegraphics[width=0.44\textwidth, trim= 75mm 33mm 51mm 32mm,
  clip=true,
  keepaspectratio=false]{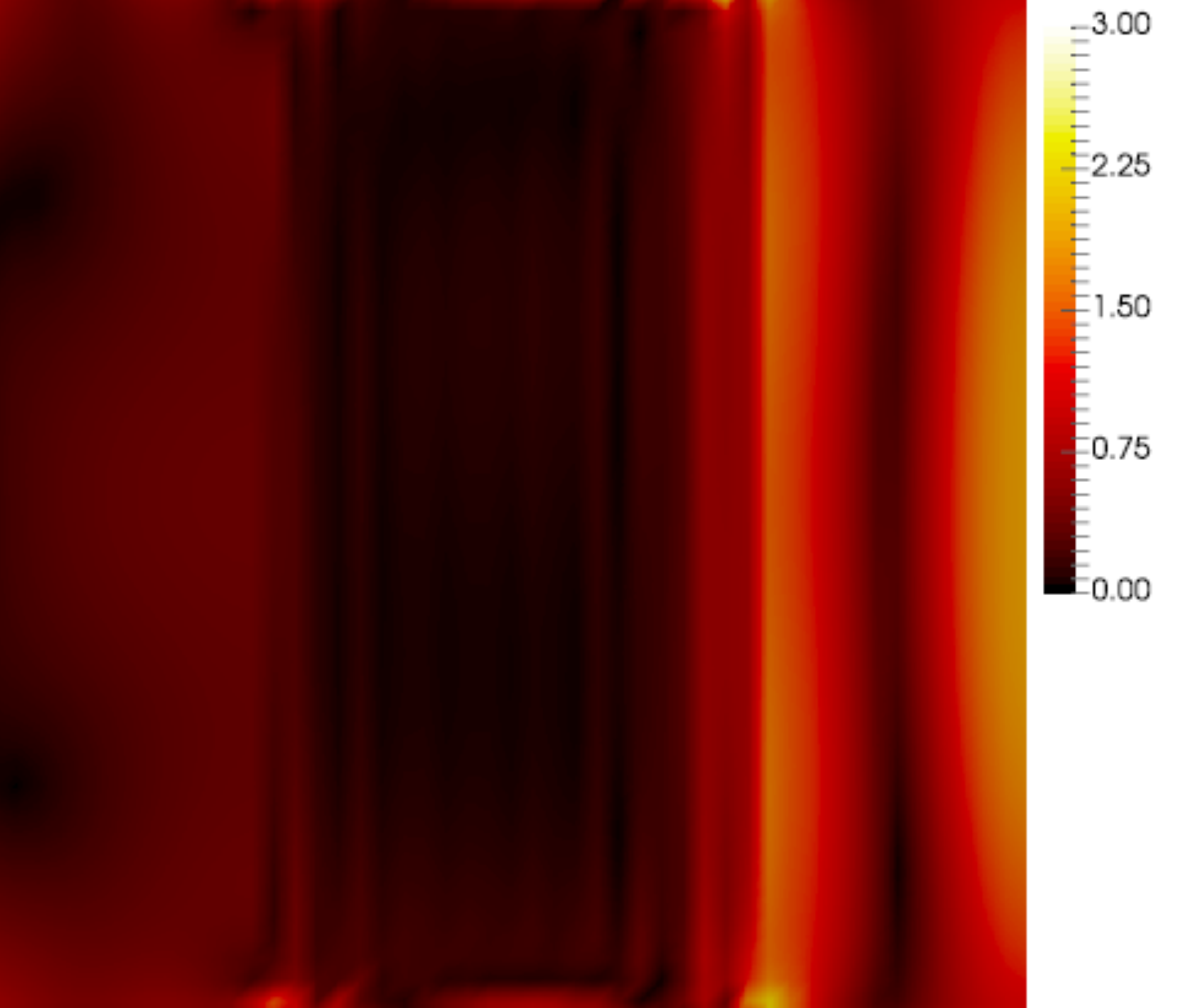}
 \begin{tikzpicture}[scale=1.5, overlay, xshift=30ex, yshift=6.5ex]
  \coordinate (A) at (-0.5, -0.5, 0.5);
  \coordinate (B) at (0.5,-0.5,0.5);
  \coordinate (C) at (0.5,0.5,0.5);
  \coordinate (D) at (-0.5, 0.5, 0.5);

  \coordinate (E) at (-0.5, 0.5, -0.5);
  \coordinate (F) at (0.5, 0.5, -0.5);
  \coordinate (G) at (0.5, -0.5, -0.5);
  \coordinate (H) at (-0.5, -0.5, -0.5);

  \coordinate (M1) at (-0.5, -0.5, 0.15);
  \coordinate (M2) at (-0.5, -0.5, -0.15);
  \coordinate (M3) at (-0.5, 0.5, -0.15);
  \coordinate (M4) at (-0.5, 0.5, 0.15);
  \coordinate (M5) at (0.5, 0.5, 0.15);
  \coordinate (M6) at (0.5,0.5,-0.15);
  \coordinate (M7) at (0.5, -0.5, -0.15);
  \coordinate (M8) at (0.5, -0.5, 0.15);

  \fill[gray!20!white, opacity=.5] (A) -- (B) -- (C) -- (D) -- cycle;
  \fill[gray!20!white, opacity=.5] (E) -- (F) -- (G) --(H) -- cycle;
  \fill[gray!20!white, opacity=.5] (D) -- (E) -- (H) -- (A)-- cycle;
  \fill[gray!20!white, opacity=.5] (B) -- (C) -- (F) -- (G) -- cycle;

  \fill[gray!75!white, opacity=.9] (M4) -- (M5) -- (M6) -- (M3) -- cycle;
  \fill[gray!75!white, opacity=.9] (M1) -- (M4) -- (M5) -- (M8) -- cycle;
  \fill[gray!75!white, opacity=.9] (M2) -- (M3) -- (M6) -- (M7) -- cycle;
  \fill[gray!75!white, opacity=.9] (M1) -- (M2) -- (M7) -- (M8) -- cycle;

  \draw[] (A) -- (B) -- (C) -- (D) --cycle (E) -- (D) (F) -- (C) (G) -- (B);
  \draw[] (E) -- (F) -- (G) ;
  \draw[densely dashed] (E) -- (H) (H) -- (G) (H) -- (A);

  \draw[dashed] (M1) -- (M4) (M2) -- (M3) ;
  \draw[] (M4) -- (M5) -- (M6) -- (M3) -- cycle;
  \draw[dashed] (M5) -- (M8) (M7) -- (M6);
  \draw[dashed] (M1) -- (M8)  (M7) -- (M2);

  \draw[->] ( -.4, -.5, .5)--( 0, -.5, .5 );
  \draw[->] (-.4, -.5, .5) -- (-.4, -.1, .5);
  \draw[->] (-.4, -.5, .5) -- (-.4, -.5, .2);

  \node[] at (0,0,0) {$\Sigma_1$};
\node[] at (-0, -.6, .5){$\e_1$};
  \node[] at (-.28, -.1, .5){$\e_3$};
  \node[] at (-.4, -.6, .-.1){$\e_2$};
\end{tikzpicture}
\caption{Metal plate $\Sigma_3$. The colors indicate the magnitude
  of $\Re(\Heff)$ in the plane $x_3=0.5$. Left: The $H$-field is
  $\e_3$-polarized. The analysis (PC) predicts transmission, the
  analysis (HC) cannot exclude transmission. Right: The $H$-field is
  $\e_2$-polarized. The analysis (PC and HC) predicts that no
  transmission is possible.}
\label{fig:Sigma3-homrefsol}
\end{figure}

As in Section \ref{sec:metal-plate}, we choose a metal plate
perpendicular to $\e_2$ of width $0.5$, i.e.\
$\Sigma_3=(0,1)\times (0.25, 0.75)\times (0,1)$.  Discretising the
cell problems with mesh size $h_Y=\sqrt{3}\cdot 1/24$, we obtain---up
to numerical errors---the effective material parameters as diagonal
matrices with
\begin{align*}
\widehat{\varepsilon^{-1}}&\approx \diag(10^{-4}, 0.5, 10^{-4})\,,\\
\Re\hat{\mu}&\approx\diag (0.228303, -0.044672, 0.228303)\,.
\end{align*}
Although we consider high-contrast media, this correspond
astonishingly well to the analytical results for perfect conductors of
Section \ref{sec:metal-plate}: The structure of the matrices agrees
and the non-zero value of $\widehat{\varepsilon^{-1}}=|Y\setminus
\overline{\Sigma}|$ is as expected from the theory of perfect
conductors. Due to the contributions of the inclusions, the values of
$\mu_{\hom}$ are different from the case of perfect conductors.

Section \ref{sec:metal-plate} shows that, for perfect conductors, only
an $H$-field polarized in $\e_3$-direction can be transmitted through
the meta-material.  Our numerical experiments allow a similar
observation for high-contrast media in Fig.~%
\ref{fig:Sigma3-homrefsol}: The homogenized reference solution only
shows a non-negligible intensity in the domain $Q_L=\{x\in G|x_3\leq
0.25\}$ left of the scatterer if the incident wave is polarized in
$\e_3$ direction.  Note that we have some reflections from the
boundary in Fig.~\ref{fig:Sigma3-homrefsol} since we do not use
perfectly matched layers as boundary conditions.  The observed
transmission properties for high-contrast media are in accordance with
the theory in Section \ref{sec:high-contrast-media}: For an
$\e_3$-polarized $H$-field as in the left figure, we cannot expect a
(weak) convergence to zero.  This corresponds to the observed
non-trivial transmission.  By contrast, in the right figure, the
$H$-field is $\e_2$-polarized and no transmission can be
observed. This corresponds to the analysis of Section
\ref{sec:parall-electr-field}, which shows that $H^\eta$ converges to
zero, weakly in $L^2(Q_M)$.

\subsection{Air cuboid $\tilde \Sigma_4$}

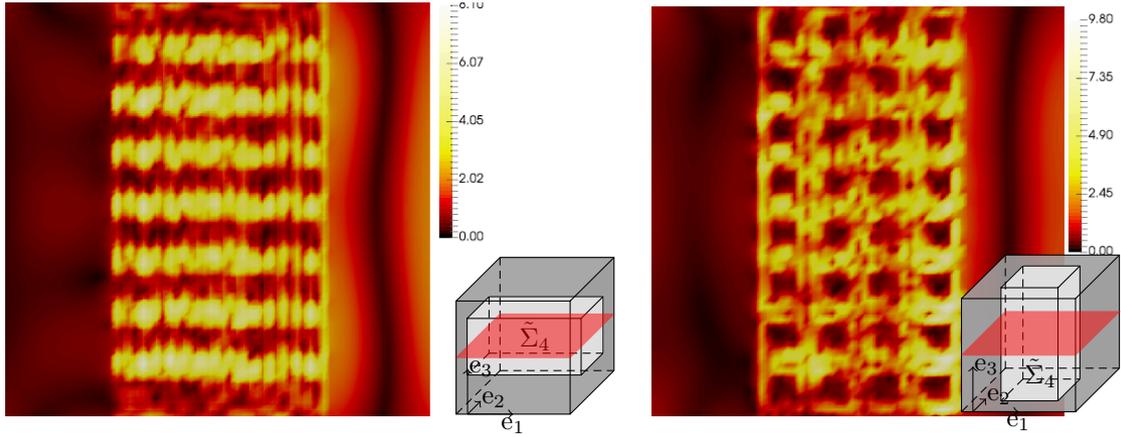
\begin{figure}
 \includegraphics[width=0.44\textwidth, trim= 75mm 33mm 51mm 32mm,
  clip=true,
  keepaspectratio=false]{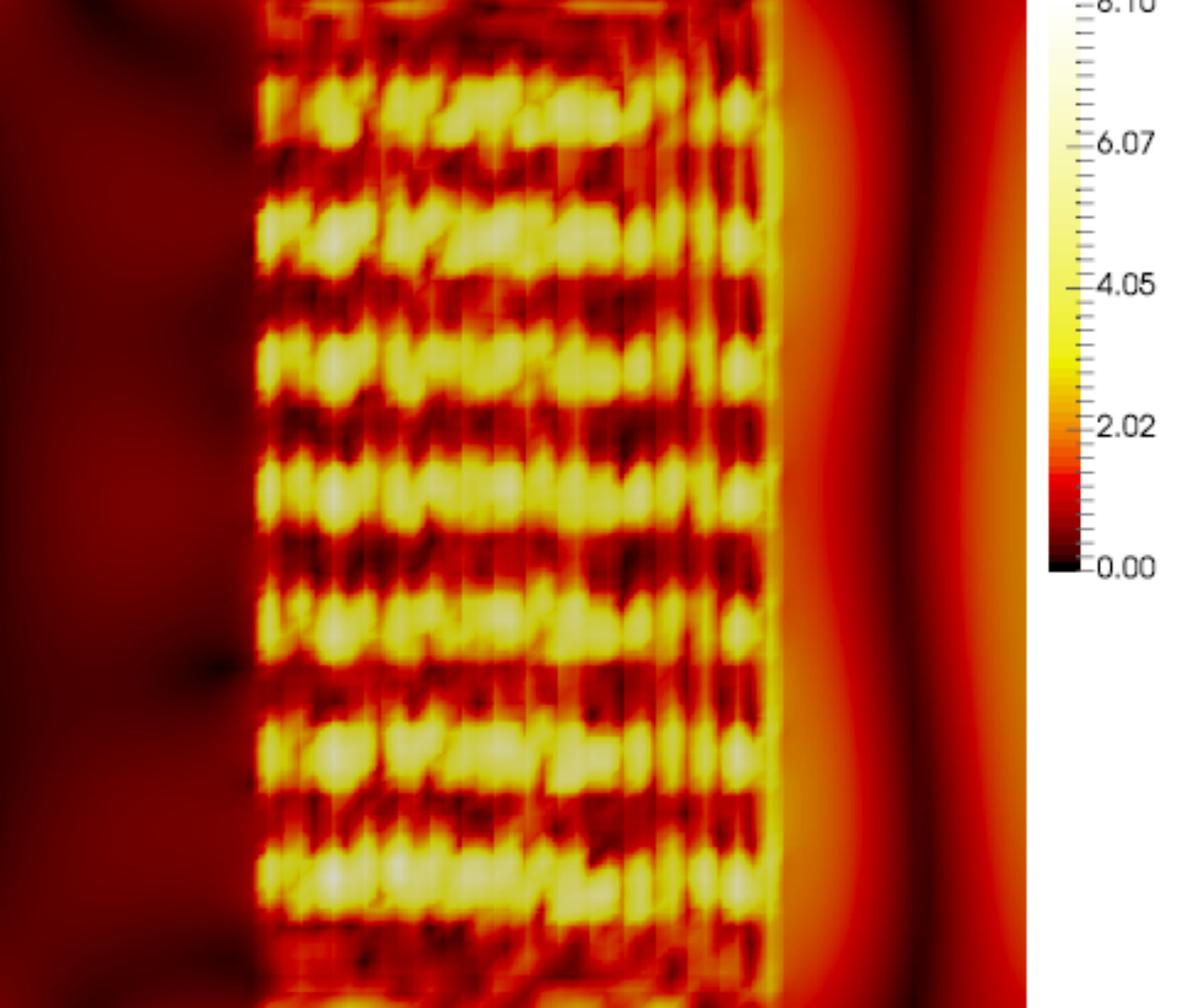}
\hspace{1.9cm}%
 \includegraphics[width=0.44\textwidth, trim= 75mm 33mm 51mm 32mm,
  clip=true,
  keepaspectratio=false]{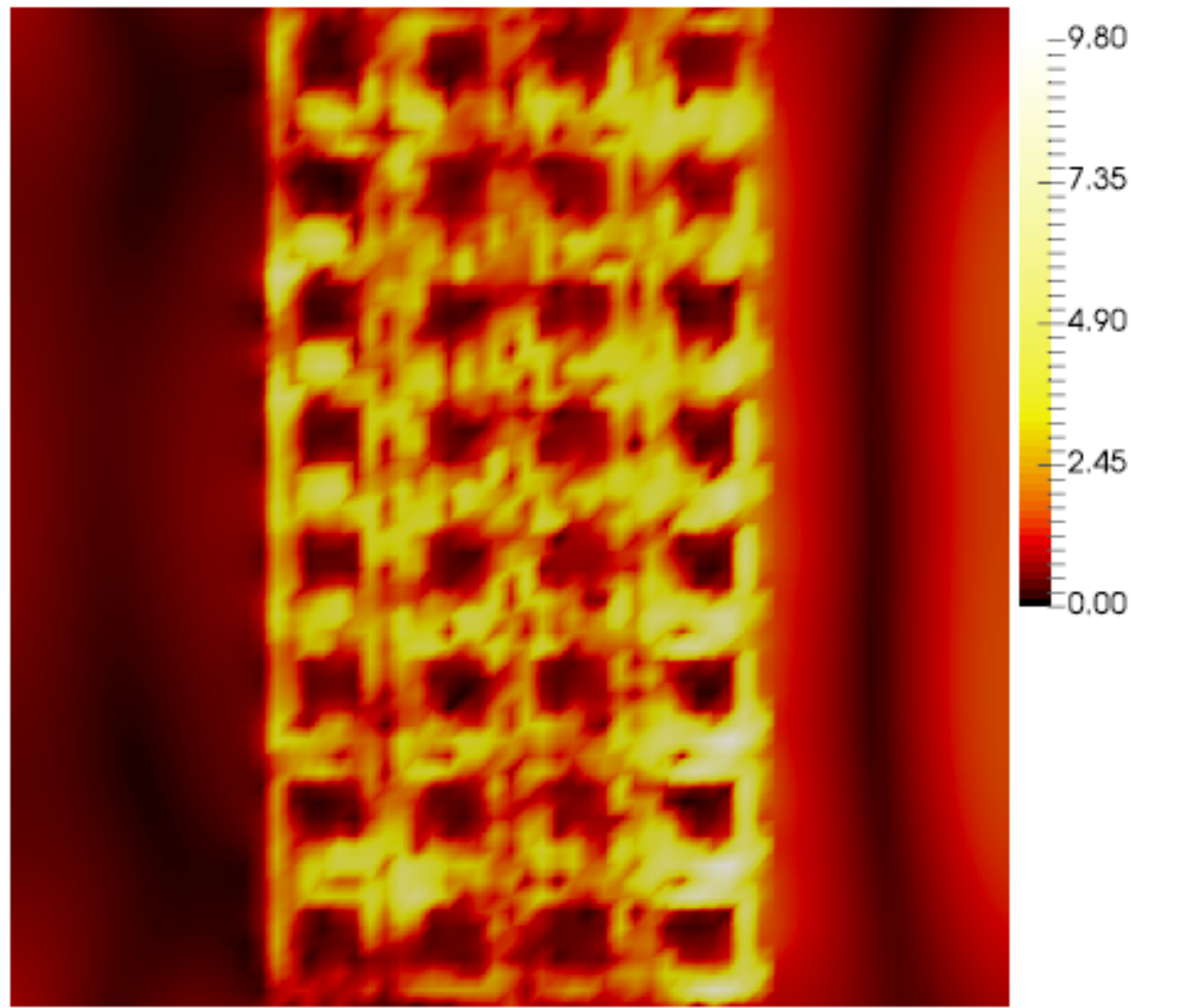}%
\begin{tikzpicture}[scale=1.5, overlay, xshift=-35ex, yshift=4.7ex]
  \coordinate (A) at (-0.5, -0.5, 0.5);
  \coordinate (B) at (0.5,-0.5,0.5);
  \coordinate (C) at (0.5,0.5,0.5);
  \coordinate (D) at (-0.5, 0.5, 0.5);

  \coordinate (E) at (-0.5, 0.5, -0.5);
  \coordinate (F) at (0.5, 0.5, -0.5);
  \coordinate (G) at (0.5, -0.5, -0.5);
  \coordinate (H) at (-0.5, -0.5, -0.5);

  \coordinate (A1) at (-0.5, -0.25, 0.25);
  \coordinate (B1) at (0.5,-0.25,0.25);
  \coordinate (C1) at (0.5,0.25,0.25);
  \coordinate (D1) at (-0.5, 0.25, 0.25);

  \coordinate (E1) at (-0.5, 0.25, -0.25);
  \coordinate (F1) at (0.5, 0.25, -0.25);
  \coordinate (G1) at (0.5, -0.25, -0.25);
  \coordinate (H1) at (-0.5, -0.25, -0.25);

  \fill[gray!75!white, opacity=.9] (A) -- (B) -- (C) -- (D) -- cycle;
  \fill[gray!75!white, opacity=.9] (E) -- (F) -- (G) --(H) -- cycle;
  \fill[gray!75!white, opacity=.9] (D) -- (E) -- (H) -- (A)-- cycle;
  \fill[gray!75!white, opacity=.9] (B) -- (C) -- (F) -- (G) -- cycle;

  \fill[gray!25!white] (A1)--(B1)--(C1)--(D1)--cycle;
  \filldraw[gray!25!white] (E1)--(F1)--(G1)--(H1)--cycle;
  \filldraw[gray!25!white] (D1)--(E1)--(H1)--(A1)--cycle;
  \filldraw[gray!25!white] (B1)--(C1)--(F1)--(G1)--cycle;

  \draw[] (A) -- (B) -- (C) -- (D) --cycle (E) -- (D) (F) -- (C) (G) -- (B);
  \draw[] (E) -- (F) -- (G) ;
  \draw[densely dashed] (E) -- (H) (H) -- (G) (H) -- (A);

  \draw[] (A1)--(B1)--(C1)--(D1)--cycle (E1)--(D1) (F1)--(C1) (G1)--(B1);
  \draw[] (E1)--(F1)--(G1);
  \draw[densely dashed] (E1)--(H1) (H1)--(G1) (H1)--(A1);

  \draw[->] ( -.4, -.5, .5)--( 0, -.5, .5 );
  \draw[->] (-.4, -.5, .5) -- (-.4, -.1, .5);
  \draw[->] (-.4, -.5, .5) -- (-.4, -.5, .2);

  \node[] at (0,0,0) {$\tilde \Sigma_4$};
\node[] at (-0, -.6, .5){$\e_1$};
  \node[] at (-.28, -.1, .5){$\e_3$};
  \node[] at (-.4, -.6, .-.1){$\e_2$};

  \filldraw[red, opacity=.5] (-0.5, 0, -0.5) -- (0.5, 0, -0.5) --(0.5, 0, 0.5) --(-0.5, 0, 0.5)--cycle;
\end{tikzpicture}
\begin{tikzpicture}[scale=1.5, overlay, xshift=60ex, yshift=6.7ex]
  \coordinate (A) at (-0.5, -0.5, 0.5);
  \coordinate (B) at (0.5,-0.5,0.5);
  \coordinate (C) at (0.5,0.5,0.5);
  \coordinate (D) at (-0.5, 0.5, 0.5);

  \coordinate (E) at (-0.5, 0.5, -0.5);
  \coordinate (F) at (0.5, 0.5, -0.5);
  \coordinate (G) at (0.5, -0.5, -0.5);
  \coordinate (H) at (-0.5, -0.5, -0.5);

  \coordinate (A1) at (-0.25, -0.5, 0.25);
  \coordinate (B1) at (0.25,-0.5,0.25);
  \coordinate (C1) at (0.25,0.5,0.25);
  \coordinate (D1) at (-0.25, 0.5, 0.25);

  \coordinate (E1) at (-0.25, 0.5, -0.25);
  \coordinate (F1) at (0.25, 0.5, -0.25);
  \coordinate (G1) at (0.25, -0.5, -0.25);
  \coordinate (H1) at (-0.25, -0.5, -0.25);

  \fill[gray!75!white, opacity=.9] (A) -- (B) -- (C) -- (D) -- cycle;
  \fill[gray!75!white, opacity=.9] (E) -- (F) -- (G) --(H) -- cycle;
  \fill[gray!75!white, opacity=.9] (D) -- (E) -- (H) -- (A)-- cycle;
  \fill[gray!75!white, opacity=.9] (B) -- (C) -- (F) -- (G) -- cycle;

  \fill[gray!25!white] (A1)--(B1)--(C1)--(D1)--cycle;
  \filldraw[gray!25!white] (E1)--(F1)--(G1)--(H1)--cycle;
  \filldraw[gray!25!white] (D1)--(E1)--(H1)--(A1)--cycle;
  \filldraw[gray!25!white] (B1)--(C1)--(F1)--(G1)--cycle;

  \draw[] (A) -- (B) -- (C) -- (D) --cycle (E) -- (D) (F) -- (C) (G) -- (B);
  \draw[] (E) -- (F) -- (G) ;
  \draw[densely dashed] (E) -- (H) (H) -- (G) (H) -- (A);

  \draw[] (A1)--(B1)--(C1)--(D1)--cycle (E1)--(D1) (F1)--(C1) (G1)--(B1);
  \draw[] (E1)--(F1)--(G1);
  \draw[densely dashed] (E1)--(H1) (H1)--(G1) (H1)--(A1);

  \draw[->] ( -.4, -.5, .5)--( 0, -.5, .5 );
  \draw[->] (-.4, -.5, .5) -- (-.4, -.1, .5);
  \draw[->] (-.4, -.5, .5) -- (-.4, -.5, .2);

  \node[] at (0,-0.35,0) {$\tilde{\Sigma}_4$};
\node[] at (-0, -.6, .5){$\e_1$};
  \node[] at (-.28, -.1, .5){$\e_3$};
  \node[] at (-.4, -.6, .-.1){$\e_2$};

  \filldraw[red, opacity=.5] (-0.5, 0, -0.5) -- (0.5, 0, -0.5) --(0.5, 0, 0.5) --(-0.5, 0, 0.5)--cycle;
\end{tikzpicture}
\caption{Metal block with holes. Left: The structure $\tilde
  \Sigma_4$, we plot the magnitude of $\Re(H^\eta)$ in the plane
  $x_3=0.545$ for $\e_3$-polarized incoming $H$-field. The analysis
  (PC) predicts no transmission, the analysis (HC) cannot exclude
  transmission. Right: A geometry in which the cylinders $\tilde
  \Sigma_4$ are rotated in $\e_3$-direction.  We plot the magnitude of
  $\Re(H^\eta)$ in the plane $x_3=0.5$ for $\e_3$-polarized incoming
  $H$-field.  Small pictures show the microstructures in the unit cube
  and the visualization planes in red.}
\label{fig:Sigma4-refsol}
\end{figure}

As with the metal cylinder, we equip the air cylinder of Section
\ref{sec:air-cylinder} with a square base in order to have a
geometry-fitting mesh.  To be precise, we define the microstructure
$\tilde \Sigma_4=(0,1)^3\setminus((0,1)\times(0.25, 0.75)^2)$.  The
effective permittivity $\widehat{\varepsilon^{-1}}$ vanishes almost
identically for this setting; numerically we obtain only entries of
order $10^{-5}$ for a discretisation of the corresponding cell problem
with mesh size $h_Y=\sqrt{3}\cdot 1/24$.  As discussed in Section
\ref{sec:air-cylinder}, no transmission through this meta-material is
expected for the high conductors.  We observe the same for
high-contrast media in Fig.~\ref{fig:Sigma4-refsol}: The (true)
reference solution (almost) vanishes in the left part $Q_L$ in all
situations. Here, we only depict $\e_3$-polarized incident waves, once
for $\tilde \Sigma_4$ as described and once for the rotated air cuboid
with main axis in $\e_3$-direction (this is the setting of Section
\ref{sec:metal-cylinder} with interchanged roles of metal and air).
Note that inside the microstructure, high intensities and amplitudes
of the $H^\eta$-field occur due to resonances in the high-contrast
medium.

\section*{Conclusion}
We analyzed the transmission properties of meta-materials consisting
of perfect conductors or high-contrast materials.  Depending on the
geometry of the microstructure, certain entries in the effective
material parameters vanish, which induces that also certain components
of the solution vanish.  This influences the transmission properties
of the material. Transmission is possible only for certain
polarizations of the incoming wave.  For perfect conductors, we
derived closed formulas for the reflection and transmission
coefficients.  Using the Heterogeneous Multiscale Method, the
homogenized solution as well as some features of the exact solution
can be approximated on rather coarse meshes and, in particular, with a
cost that is independent of the periodicity length.  Our numerical
experiments of three representative geometries with high-contrast
materials confirm the theoretical predictions of their transmission
properties.

\bibliographystyle{abbrv}

\end{document}